\documentclass[12pt, 
]{amsart}

\usepackage[all]{xy}
\usepackage{mathrsfs, amsthm, amscd, amsmath, amssymb, graphicx}

\usepackage[dvipdfmx]{hyperref}

\newtheorem{theo}{Theorem}[section]
\newtheorem{prop}[theo]{Proposition}
\newtheorem{lemm}[theo]{Lemma}
\newtheorem{cor}[theo]{Corollary}

\newtheorem{prob}[theo]{Problem}
\newtheorem{conj}[theo]{Conjecture}

\theoremstyle{definition}
\newtheorem{defi}[theo]{Definition}
\newtheorem{ex}[theo]{Example}
\newtheorem{step}[theo]{Observation}

\theoremstyle{remark}
\newtheorem{rem}[theo]{Remark}

\makeatletter
    
    \@addtoreset{equation}{section}
\makeatother

\newcommand{\Ker}[0]{\operatorname{Ker}}
\newcommand{\Image}[0]{\operatorname{Im}}
\newcommand{\Supp}[0]{\operatorname{Supp}}

\newcommand{\deldel}{\sqrt{-1}\partial \overline{\partial}}
\newcommand{\dbar}{\overline{\partial}}
\newcommand{\e}{\varepsilon}
\newcommand{\ome}{\widetilde{\omega}}
\newcommand{\I}[1]{\mathcal{I}(#1)}
\newcommand{\OX}{\mathcal{O}}

\newcommand{\lla}[0]{{\langle\!\hspace{0.02cm} \!\langle}}
\newcommand{\rra}[0]{{\rangle\!\hspace{0.02cm}\!\rangle}}

\begin{document}

\title[]
{A transcendental approach to \\ injectivity theorem for log canonical pairs}

\author{Shin-ichi MATSUMURA}
\address{Mathematical Institute, Tohoku University, 
6-3, Aramaki Aza-Aoba, Aoba-ku, Sendai 980-8578, Japan.}
\email{{\tt mshinichi@m.tohoku.ac.jp, mshinichi0@gmail.com}}
\date{\today, version 0.01}
\renewcommand{\subjclassname}{%
\textup{2010} Mathematics Subject Classification}

\maketitle

\begin{abstract}
In this paper, 
we study transcendental aspects of the cohomology groups of 
adjoint bundles of log canonical pairs, 
aiming to establish an analytic theory for log canonical singularities. 
As a result, in the case of purely log terminal pairs, 
we give an analytic proof of the injectivity theorem 
originally proved by the Hodge theory. 
Our method is based on the theory of harmonic integrals  
and the $L^2$-method for the $\dbar$-equation, 
and it enables us to generalize the injectivity theorem 
to the complex analytic setting. 
\end{abstract}

%\tableofcontents

\section{Introduction}\label{Sec1}

In the study of the vanishing theorem, 
analytic methods and algebraic geometric methods 
have been nourishing each other in the last decades. 
The injectivity theorem is one of the most important generalizations 
of the Kodaira vanishing theorem and plays an important role 
when we study fundamental problems in higher dimensional algebraic geometry. 
We recently obtained satisfactory injectivity theorems formulated by multiplier ideal sheaves 
and their applications to vanishing theorems and extension problems of (holomorphic) sections 
from the analytic viewpoint (see \cite{FM16}, \cite{GM13}, and \cite{Mat16}). 
Our results can be seen as an analytic theory of the injectivity theorem 
for kawamata log terminal ({\textit{klt}}, for short) singularities. 
The next interesting problem is the study of the injectivity theorem  
for log canonical ({\textit{lc}}, for short) singularities.

The following result is an injectivity theorem for lc pairs, 
whose proof heavily depends on the Hodge theory 
(see \cite{Amb03}, \cite{Amb14}, \cite{EV92}, \cite[Section 6]{Fuj11}, \cite{Fuj12b}, and \cite{Fuj13b} for the Hodge theoretic viewpoint). 
It is one of the most important problems in complex geometry 
to establish an analytic theory for lc singularities. 
In this paper, 
we study a transcendental approach to 
the following result and lc singularities.

\begin{theo}\label{Fujino}
Let $D$ be a simple normal crossing divisor 
on a smooth projective variety $X$ 
and $F$ be a semi-ample line bundle on $X$. 
Let $s$ be a $($holomorphic$)$ section of a positive multiple $F^m$ 
such that the zero locus $s^{-1}(0)$ 
contains no lc centers of the lc pair $(X,D)$. 
Then, the multiplication map induced by the tensor product with $s$ 
\begin{equation*}
H^{q}(X, K_{X} \otimes D \otimes F) 
\xrightarrow{\otimes s} 
H^{q}(X, K_{X} \otimes D \otimes F^{m+1} )
\end{equation*}
is injective for every $q$. 
Here $K_{X}$ denotes the canonical bundle of $X$. 
\end{theo}

Let $D$ be a simple normal crossing divisor on a smooth variety $X$ 
and let $D = \sum_{i \in I} D_{i}$ be the irreducible decomposition of $D$. 
Then an irreducible component of $D_{i_{1}}\cap \dots \cap D_{i_{k}}$ 
is called an {\textit{lc center}} of the pair $(X,D)$. 
Note that we interchangeably use the words 
\lq \lq (Cartier) divisors", \lq \lq (holomorphic) line bundles", 
and \lq \lq invertible sheaves".

The celebrated injectivity theorem proved by Koll\'ar (see \cite{Kol86a}) 
is the special case of $D=0$ of the above result, 
and thus the above result can be seen as a generalization of 
Koll\'ar's injectivity theorem to lc pairs. 
On the other hand, 
Enoki gave an analytic proof of Koll\'ar's injectivity theorem 
and generalized it to semi-positive line bundles 
on compact K\"ahler manifolds (see \cite{Eno90}). 
Here a (holomorphic) line bundle $F$ is said to be 
{\textit{semi-positive}} if 
$F$ admits a smooth hermitian metric $h$ 
such that the (Chern) curvature $\sqrt{-1}\Theta_{h}(F)$ 
is a semi-positive $(1,1)$-form.  
We remark that 
semi-ample line bundles are always semi-positive, but the converse is not true.
In summary, we have the following diagram\,:
\vspace{0.2cm}
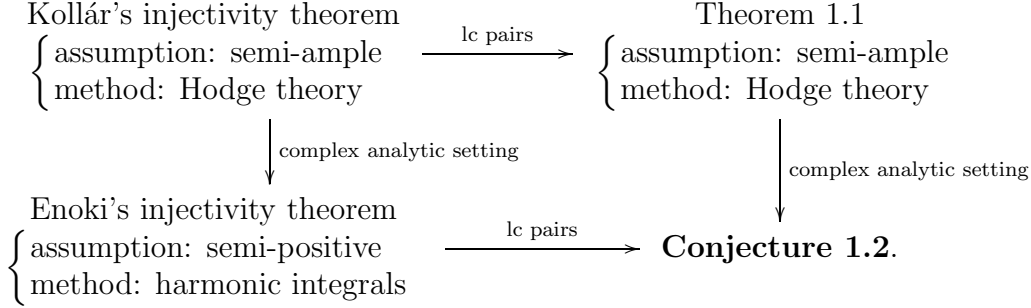
\[\xymatrixcolsep{4pc}\xymatrix{
\hspace{-1.5cm} \text{$
\begin{array}{cl}
\text{Koll\'ar's injectivity theorem} \\
\left \{\hspace{-0.5cm} 
\begin{array}{cl}
&\text{assumption: semi-ample}\\
&\text{method: Hodge theory} 
\end{array}
\right.
\end{array}
$}
\ar[d]^-{\text{complex analytic setting}} \ar[r]^-{\text{lc pairs}} &
\text{$\begin{array}{cl}
\text{Theorem \ref{Fujino}}\\
\left \{\hspace{-0.5cm} 
\begin{array}{cl}
&\text{assumption: semi-ample}  \\
&\text{method: Hodge theory} 
\end{array}
\right.
\end{array}
$}
\ar[d]^-{\text{complex analytic setting}}\\
\hspace{-1.5cm} 
\text{$
\begin{array}{cl}
\text{Enoki's injectivity theorem}\\
\left \{\hspace{-0.5cm} 
\begin{array}{cl}
&\text{assumption: semi-positive}\\
&\text{method: harmonic integrals} 
\end{array}
\right.
\end{array}
$} \ar[r]^-{\text{lc pairs}}
&\text{$
\begin{array}{cl}
\text{{\bf{Conjecture \ref{F-conj}}}}.\\
\end{array}
$}}
\]
\vspace{0.3cm}

Therefore it is natural to ask 
whether we can give an analytic proof of Theorem \ref{Fujino} and 
generalize Theorem \ref{Fujino} to the complex analytic setting. 
In this paper, we study the following conjecture posed in \cite{Fuj15b}. 
By using higher direct image sheaves under projective (or K\"ahler) morphisms, 
we can formulate a relative version of the injectivity theorem. 
We have already obtained a relative version of Enoki's injectivity theorem in \cite{Tak95} and  
its generalization for klt singularities in \cite{Mat16b}. 
A relative version of Theorem \ref{Fujino} is one of the important open problems 
on vanishing theorems in the minimal model program (see \cite[Problem 1.8]{Fuj13a}). 
The following conjecture is a first step to consider \cite[Problem 1.8]{Fuj13a} 
from the analytic viewpoint.

\begin{conj}[{\cite[Conjecture 2.21]{Fuj15b}, cf. \cite[Problem 1.8]{Fuj13a}}] 
\label{F-conj}
Let $D$ be a simple normal crossing divisor on a compact K\"ahler manifold $X$ 
and $F$ be a semi-positive line bundle on $X$. 
Let $s$ be a $($holomorphic$)$ section of a positive multiple $F^m$ 
such that the zero locus $s^{-1}(0)$ 
contains no lc centers of the lc pair $(X,D)$. 
Then, the multiplication map induced by the tensor product with $s$ 
\begin{equation*}
H^{q}(X, K_{X} \otimes D \otimes F) 
\xrightarrow{\otimes s} 
H^{q}(X, K_{X} \otimes D \otimes F^{m+1} )
\end{equation*}
is injective for every $q$. 
\end{conj}

The following theorem, which is one of the main results of this paper, 
can be seen as a generalization of Enoki's result to 
purely log terminal ({\textit{plt}}, for short) pairs. 
See \cite{Ohs04} and \cite{Fuj12a} for the formulation of Theorem \ref{main} 
and see \cite{FM16} for applications of this formulation.  
The proof of Theorem \ref{main}, 
which is based on the theory of harmonic integrals and the $L^2$-method, 
provides an analytic method to study lc singularities. 
As a corollary of Theorem \ref{main}, 
we obtain a partial answer for Conjecture \ref{F-conj}, 
which completely solves Conjecture \ref{F-conj} in the case of plt pairs. 

\begin{theo}[Main Theorem]\label{main}
Let $D$ be a simple normal crossing divisor on a compact K\"ahler manifold $X$. 
Let $F$ $($resp. $M$$)$ be a $($holomorphic$)$ line bundle on $X$ 
with a smooth hermitian metric $h_{F}$  $($resp. $h_{M}$$)$ 
such that 
\begin{equation*}
\sqrt{-1}\Theta_{h_M}(M)\geq 0 \  \text{ and }\ 
\sqrt{-1}(\Theta_{h_F}(F)-t \Theta 
_{h_M}(M))\geq 0 
\ \text{ for some }t>0. 
\end{equation*}
We assume that the pair $(X,D)$ is a plt pair 
$($that is, $D_{i} \cap D_{j} = \emptyset$ for $i \not = j$ 
for the irreducible decomposition $D = \sum_{i\in I}D_{i}$$)$. 
Let $s$ be a $($holomorphic$)$ section of $M$  
such that the zero locus $s^{-1}(0)$ 
contains no lc centers of the lc pair $(X,D)$.
Then, the multiplication map induced by the tensor product with $s$
\begin{equation*}
H^q(X, K_X\otimes D \otimes F)
\xrightarrow{\otimes s } 
H^q(X, K_X\otimes D \otimes F\otimes M)
\end{equation*} 
is injective for every $q$. 
\end{theo}

\begin{cor}\label{main-co}
Under the same situation as in Conjecture \ref{F-conj}, 
we assume that the pair $(X,D)$ is a plt pair. 
Then, the same conclusion as in Conjecture \ref{F-conj} holds,  
that is, the multiplication map induced by the tensor product with $s$ 
\begin{equation*}
H^{q}(X, K_{X} \otimes D \otimes F) 
\xrightarrow{\otimes s} 
H^{q}(X, K_{X} \otimes D \otimes F^{m+1} )
\end{equation*}
is injective for every $q$.

In particular, 
Conjecture \ref{F-conj} is affirmatively solved 
for a plt pair $(X,D)$. 
\end{cor}

\begin{rem}\label{main-rem}
Let $D$ be a simple normal crossing divisor on a smooth variety $X$ and 
$D = \sum_{i\in I}D_{i}$ be the irreducible decomposition of $D$. 
Then we have\,: 
\begin{itemize}
\item[$\bullet$] $(X,D)$ is called a plt pair 
if $D_{i} \cap D_{j} = \emptyset$ for $i \not = j$. 
\item[$\bullet$] When $(X,D)$ is a plt pair, 
the zero locus $s^{-1}(0)$ contains no lc centers of $(X,D)$ if and only if 
$s^{-1}(0)$ does not contain $D_{i}$ for every $i \in I$. 
\end{itemize}
\end{rem}

In order to explain the difficulties of Conjecture \ref{F-conj} and 
the proof of Theorem \ref{main}, 
we recall klt singularities 
and multiplier ideal sheaves. 
The notion of multiplier ideal sheaves plays an important role 
in the recent developments in algebraic geometry. 
The multiplier ideal sheaf $\mathcal{J}(X,D)$, 
which is algebraically defined for a log pair $(X, D)$, 
can be  seen as  a \lq \lq non-klt" ideal (see subsection \ref{Sec2-1}). 
This is because the pair $(X, D)$ has klt singularities 
if and only if $\mathcal{J}(X,D)$ coincides with the structure sheaf $\mathcal{O}_{X}$. 
On the other hand, the multiplier ideal sheaf 
can be analytically defined for singular hermitian metrics 
in terms of the $L^2$-integrability of holomorphic functions (see Definition \ref{multi}). 
Thanks to the analytic expression of multiplier ideal sheaves, 
we can treat klt singularities by using the $L^2$-method. 
Indeed, we have already proved various injectivity theorems  
with multiplier ideal sheaves, 
which can be seen as an injectivity theorem for klt singularities  
(see \cite{Fuj15b} and \cite{Mat15} for the recent developments).

The notion of \lq \lq  non-lc" ideal sheaves has already introduced in \cite{FST11} and \cite{Fuj10}. 
However we have no analytic interpretations for non-lc ideal sheaves (see \cite[Question 2.4]{Fuj10}). 
Although we can apply the $L^2$-method for klt pairs as mentioned above,  
the usual $L^2$-method does not work for lc pairs 
since lc singularities are worse than klt singularities. 
This is one of the difficulties of Corollary \ref{F-conj}. 
To overcome this difficulty, in the proof of Theorem \ref{main}, 
we estimate the order of divergence of suitable $L^2$-norms  
(that is, how far from klt singularities). 
This argument may provide a new technique to treat lc singularities 
from the analytic viewpoint.

For the proof of Theorem \ref{main}, 
we reduce Theorem \ref{main} to the following theorem. 
In this reduction step, we use the assumption that $(X,D)$ is a plt pair. 
However, we emphasize that we do not need this assumption in Theorem \ref{key}.

\begin{theo}[Key Result]\label{key}
Let $D$ be a simple normal crossing divisor on a compact K\"ahler manifold $X$. 
Let $F$ $($resp. $M$$)$ be a $($holomorphic$)$ line bundle on $X$ 
with a smooth hermitian metric $h_{F}$  $($resp. $h_{M}$$)$ 
such that 
\begin{equation*}
\sqrt{-1}\Theta_{h_M}(M)\geq 0 \  \text{ and }\ 
\sqrt{-1}(\Theta_{h_F}(F)-t \Theta 
_{h_M}(M))\geq 0 
\ \text{ for some }t>0. 
\end{equation*}
We consider the map 
\begin{equation*}
\Phi_{D} : 
H^{q}(X, K_{X} \otimes F) 
\xrightarrow{\quad} 
H^{q}(X, K_{X} \otimes D \otimes F )
\end{equation*}
induced by the natural inclusion $\OX_{X} \hookrightarrow \OX_{X}(D)$. 
Then, the multiplication map on the image $\Image \Phi_{D}$ 
induced by the tensor product with $s$ 
\begin{equation*}
\Image \Phi_{D}
\xrightarrow{\quad \otimes s \quad } 
H^{q}(X, K_{X} \otimes D \otimes F\otimes M )
\end{equation*}
is injective for every $q$. 
\end{theo}

In the proof of Theorem \ref{key}, 
the following theorem plays an important role. 
This theorem is a refinement of the hard Lefschetz theorem 
with multiplier ideal sheaves proved in \cite{DPS01}, 
which is independently of interest.

\begin{theo}[Hard Lefschetz Theorem]\label{Lef}
Let $\omega$ be a K\"ahler form on a compact K\"ahler manifold $X$ and 
$(G, h)$ be a singular hermitian line bundle with semi-positive curvature. 
Assume that the singular hermitian metric $h$ is 
smooth on a non-empty Zariski open set in $X$. 
Then, for a harmonic $G$-valued $(n,q)$-form $u \in \mathcal{H}^{n,q}_{h,\omega}(G)$ 
with respect to $h$ and $\omega$, 
we have 
$$
* u \in H^{0}(X, \Omega_{X}^{n-q}\otimes G \otimes \I{h}), 
$$
where $\Omega_{X}^{n-q}$ is the vector bundle of holomorphic $(n-q,0)$-forms,  
$*$ is the Hodge star operator with respect to $\omega$,  
and $\I{h}$ is the multiplier ideal sheaf of $h$. 
See Section \ref{Sec2} for the definition of 
the set of harmonic forms $\mathcal{H}^{n,q}_{h,\omega}(G)$. 
\end{theo}

This paper is organized as follows\,$:$  
In Section \ref{Sec2}, 
we summarize the fundamental results, 
including singularities of pairs, multiplier ideal sheaves, the $L^2$-method, 
and the theory of harmonic integrals.  
We give a proof of Theorem \ref{Lef} (resp. Theorem \ref{key}, Theorem \ref{main}) 
in subsection \ref{Sec3-1} (resp. subsection \ref{Sec3-2}, subsection \ref{Sec3-3}). 
In subsection \ref{Sec3-4}, 
we discuss open problems related to the contents of this paper.

\subsection*{Acknowledgements}
The author wishes to express his gratitude to Professor Junyan Cao for stimulating discussions, 
and he wishes to thank Professor Osamu Fujino for giving useful comments. 
He is supported by the Grant-in-Aid for 
Young Scientists (B) $\sharp$25800051 from JSPS.

\section{Preliminaries}\label{Sec2}
In this section, 
we fix the notations and 
summarize the facts needed in this paper.

\subsection{Singularities of pairs and multiplier ideal sheaves}\label{Sec2-1}
We first recall the notion of singularities of pairs. 

\begin{defi}[Klt, plt, and lc singularities]\label{log-sing}
Let $(X, D)$ be a log pair (that is, 
a pair of a normal variety $X$ and an effective 
$\mathbb{Q}$-divisor $D$ on $X$ such that 
$K_X+D$ is $\mathbb{Q}$-Cartier, where $K_{X}$ is the canonical divisor of $X$). 
For a log resolution $\varphi:Y\rightarrow X$ of $(X,D)$, 
we define the $\mathbb{Q}$-divisor $D'$ by 
$$
K_Y+D'=\varphi^*(K_X+D),
$$ 
and we consider the irreducible decomposition of $D'=\sum b_{i}E_{i}$.  
Then  singularities of the pair $(X,D)$ is defined as follows\,: 
\begin{itemize}
\item[$\bullet$] $(X,D)$ is said to be  
{\it{kawamata log terminal}} $(${\it{klt}}, 
for short$)$ if $b_i < 1$ for every $i$. 
\item[$\bullet$] $(X,D)$ is said to be  
{\it{log canonical}} $(${\it{lc}}, for short$)$ if $b_i \leq 1$ for every $i$.
\end{itemize}
Note that the above definitions do not depend on the choice of log resolutions.

For an lc pair $(X, D)$, 
if there exists a log resolution $\varphi:Y\to X$ of $(X, D)$ such that 
the exceptional set ${\rm{Exc}} (\varphi)$ is a divisor, 
that $b_i<1$ for every $\varphi$-exceptional divisor $E_i$, and  
that $\lfloor D \rfloor$ is a sum of disjoint prime divisors, 
then the pair $(X, D)$ is said to be 
{\it{purely log terminal}} $(${\it{plt}}, for short$)$. 
Here $\lfloor D \rfloor$ denotes the divisor defined by the round-downs 
of the coefficients of $D$. 
\end{defi}

In this paper, we consider only log smooth pairs,  
and thus the following example is enough to follow the contents of this paper.

\begin{ex}\label{log-ex}
Let $(X,D)$ be a log smooth pair (that is, 
a pair of a smooth variety $X$ and an effective  
$\mathbb{Q}$-divisor $D$ on $X$ with simple normal crossing support). 
Let $D=\sum b_{i}D_{i}$ be the irreducible decomposition. 
Then, by the definition, we can easily check the following claims\,$:$   
\begin{enumerate}
\item[$\bullet$] The pair $(X,D)$ is klt if and only if $b_{i}<1$ for every $i$.  
\item[$\bullet$] The pair $(X,D)$ is plt if and only if 
$\lfloor D \rfloor$ is a sum of disjoint prime divisors. 
\item[$\bullet$] The pair $(X,D)$ is lc if and only if $b_{i}\leq 1$ for every $i$. 
\end{enumerate}
\end{ex}

In general, for a log pair $(X, D)$, 
the multiplier ideal sheaf $\mathcal{J}(X,D) \subset \mathcal{O}_{X}$ 
is defined by 
$$
\mathcal{J}(X,D):=\varphi_{*}\mathcal{O}_{Y}(-\lfloor  \varphi^{*}D -K_{Y/X} \rfloor), 
$$
where $K_{Y/X}$ is the relative canonical divisor and 
$\varphi:Y \to X$ is a log resolution of $(X,D)$. 
We remark that the multiplier ideal sheaf $\mathcal{J}(X,D)$ 
does not depend on the choice of log resolutions. 
Then the pair $(X,D)$ is klt if and only if 
$\mathcal{J}(X,D)$ coincides with $\mathcal{O}_{X}$. 
In this sense, multiplier ideal sheaves can be regarded as a non-klt ideal,  
and thus the injectivity theorem with multiplier ideal sheaves  
(proved in \cite{Mat16}, \cite{FM16}) can  be seen 
as an injectivity theorem for klt singularities. 
On the other hand, multiplier ideal sheaves can be defined for singular hermitian metrics 
(see \cite{Dem-L2} for singular hermitian metrics and curvatures).

\begin{defi}{(Multiplier ideal sheaves).}\label{multi}
Let $G$ be a $($holomorphic$)$ line bundle on a complex manifold $X$ 
and $h$ be a singular hermitian metric on $G$ such that 
$\sqrt{-1} \Theta_{h}(G) \geq \gamma$  
for some smooth $(1,1)$-form $\gamma$ on $X$. 
Then \textit{multiplier ideal sheaf} $\I{h}$ of $h$ is 
defined to be  
\begin{equation*}
\I{h}(B):=  \{f \in \mathcal{O}_{X}(B)\ \big|\ 
|f|e^{-\varphi} \in L^{2}_{\rm{loc}}(B) \}
\end{equation*}
for every open set $B \subset X$, 
where $\varphi$ is a local weight of $h$. 
\end{defi}

\begin{ex}\label{multi-ex}
For an effective divisor $D$ on a complex manifold $X$, 
let $g$ be a smooth hermitian metric on the line bundle $D$ and 
$t$ be the natural section of the effective divisor $D$.
Then the singular hermitian metric $h_{D}$ on the line bundle $D$ 
can be defined by 
$$
\varphi:=\frac{1}{2} \log (|t|^2_{g}) \quad \text{ and } \quad 
h_{D}:=ge^{-2\varphi}=\frac{1}{|t|^2}, 
$$
where $|t|_{g}$ is the point-wise norm of $t$ with respect to $g$ 
(see subsection \ref{Sec2-2}). 
Note that the singular hermitian metric 
$h_{D}$ does not depend on the choice of $g$. 
Then it is easy to see that 
the multiplier ideal $\I{h_{D}}$ of $h_{D}$ coincides with 
the multiplier ideal sheaf $\mathcal{J}(X,D)$ of the pair $(X,D)$. 
Moreover, 
when the support of $D$ is simple normal crossing, 
we can easily check $\I{h_{D}}= \mathcal{O}_{X}(-\lfloor D \rfloor)$. 
\end{ex}

\subsection{$L^2$-spaces and differential operators}\label{Sec2-2}

From now on, throughout Section \ref{Sec2},   
let $X$ be a (not necessarily compact) complex manifold of dimension $n$ and 
$G$ be a (holomorphic) line bundle on $X$. 
Further let $\omega$ be a positive $(1,1)$-form on $X$ 
(which is assumed to be a K\"ahler form in the main part of this paper) 
and $h$ be a singular hermitian metric on $G$. 
We always assume that the curvature $\sqrt{-1}\Theta_{h}(G)$ of $h$ 
satisfies $\sqrt{-1}\Theta_{h}(G) \geq \gamma$ 
for some smooth $(1,1)$-form $\gamma$. 

For $G$-valued $(p,q)$-forms $u$ and $v$, 
the notation $\langle u, v\rangle _{h, \omega}$ denotes the point-wise inner product 
with respect to $h$ and $\omega$, and 
$\lla u, v \rra  _{h, \omega}$ denotes 
the (global) inner product defined by
\begin{equation*}
\lla u, v   \rra  _{h, \omega}:=
\int_{X} 
\langle u, v  \rangle _{h, \omega}\, dV_{\omega},  
\end{equation*}
where $dV_{\omega}$ is the volume form defined 
by $dV_{\omega}:=\omega^{n}/n!$. 
(Recall that $n$ is the dimension of $X$.)  
The $L^{2}$-space of $G$-valued $(p, q)$-forms with respect to 
$h$ and $\omega$ is defined by
\begin{align*}
L_{(2)}^{p, q}(G)_{h, \omega}:=
L_{(2)}^{p, q}(X, G)_{h, \omega}:= 
\{u  \,|\,  u \text{ is a }G\text{-valued }(p, q)\text{-form with } 
\|u \|_{h, \omega}< \infty \}. 
\end{align*}
Then the maximal closed extension of the $\dbar$-operator 
determines a densely defined closed operator 
$\dbar: L_{(2)}^{p, q}(G)_{h, \omega} \to L_{(2)}^{p, q+1}(G)_{h, \omega}$ 
with the domain 
$$
{\rm{Dom}}\, \dbar:=\{u \in L_{(2)}^{p, q}(G)_{h, \omega} \, | \, 
\dbar u \in L_{(2)}^{p, q+1}(G)_{h, \omega} \}. 
$$
Strictly speaking, the closed operator $\dbar$ depends on $h$ and $\omega$
since the domain and the range depend on them, 
but we often omit the subscript  
(for example, we simply write $\dbar_{h,\omega}$ 
as $\dbar$). 
In general, we have the orthogonal decomposition 
$$
L_{(2)}^{n, q}(G)_{h, \omega}=
\overline{\Image \dbar} \oplus 
\mathcal{H}^{n,q}_{h,\omega}(G) \oplus  
\overline{\Image \dbar^{*}_{h, \omega}}, 
$$
where $\dbar^{*}_{h, \omega}$ is the Hilbert space adjoint of $\dbar$, 
the subspace ${\Image \dbar}$ (resp. $\Image \dbar^{*}_{h, \omega}$) 
is the range of $\dbar$ (resp. $\dbar^{*}_{h, \omega}$), 
and the subspace $\mathcal{H}^{n,q}_{h,\omega}(G)$ is the set of 
harmonic forms with respect to $h$ and $\omega$, 
that is, 
$$
\mathcal{H}^{n,q}_{h,\omega}(G):=
\{u \in L_{(2)}^{n, q}(G)_{h, \omega} \,|\, 
\dbar u = 0 \text{ and } \dbar^{*}_{h, \omega} u=0 
\}. 
$$
For example, see \cite[(1.2) Theorem]{Dem-L2} for the above orthogonal decomposition.

When $h$ is smooth on $X$, 
the Chern connection $D=D_{(G,h)}$ can be determined 
by the holomorphic structure of $G$ and the smooth hermitian metric $h$, 
which can be written as $D = D'_{h} + \dbar$ 
with the $(1,0)$-connection $D'_{h}$ and the $\dbar$-operator. 
The maximal closed extension of the $(1,0)$-connection $D'_{h}$ 
is also a densely defined closed operator 
$D'_{h}: L_{(2)}^{p, q}(G)_{h, \omega} \to L_{(2)}^{p+1, q}(G)_{h, \omega}$, 
whose domain is 
$$
{\rm{Dom}}\, D'_{h}:=\{u \in L_{(2)}^{p, q}(G)_{h, \omega} \, | \, 
D'_{h} u \in L_{(2)}^{p+1, q}(G)_{h, \omega} \}. 
$$

We consider the Hodge star operator $\ast$ with respect to $\omega$
$$
\ast=\ast_{\omega}\, \colon \, C_{\infty}^{p,q}(G) \to 
C_{\infty}^{n-q,n-p}(G),  
$$
where $C_{\infty}^{p,q}(G)$ is 
the set of smooth $G$-valued $(p,q)$-forms on $X$. 
By the definition, we have 
$\langle u, v\rangle _{h, \omega}\, dV_{\omega}=
u \wedge H \overline{\ast v}$ and $**u=(-1)^{\deg u} u$, 
where $H$ is a local function representing $h$. 
In this paper, 
the notations ${D_{h,\omega}'^{*}}$ and $\dbar^{*}_{h,\omega}$ denote  
the Hilbert space adjoint of $D_{h}'$ and $\dbar$. 
If $\omega$ is complete, 
the Hilbert space adjoint coincides with 
the maximal closed extension of the formal adjoint 
(for example, see \cite[(8.2) Lemma]{Dem-book}). 
In particular, when $\omega$ is complete, we have 
$$
{D_{h,\omega}'^{*}}=-\ast \dbar \ast \quad \text{and} \quad 
\dbar^{*}_{h,\omega}=-\ast {D'}_{h,\omega} \ast. 
$$
The following proposition is obtained from 
the Bochner-Kodaira-Nakano identity and the density lemma 
(see \cite{DPS01} and \cite[(1.2) Theorem]{Dem-book}).

\begin{prop}\label{BKN}
Under the same situation as the first of subsection \ref{Sec2-2}, 
we assume that $\omega$ is a complete K\"ahler form and 
$h$ is smooth on $X$. 
Then we have the following identity\,$:$
$$
[\dbar, \dbar^{*}_{h,\omega}]
=[{D_{h}'}, {D_{h,\omega}'^{*}}]
+ [\sqrt{-1}\Theta_{h}(G), \Lambda_{\omega}], 
$$
where $\Lambda_{\omega}$ is 
the adjoint operator of the wedge product $\omega \wedge \bullet$, 
and $[\bullet, \bullet]$ is the graded bracket 
defined by $[A,B]=A-(-1)^{\deg A \deg B}B$.

Moreover, for every $u \in {\rm{Dom}}\, \dbar \cap {\rm{Dom}}\,\dbar^{*}_{h,\omega} 
\subset L_{(2)}^{p, q}(G)_{h, \omega}$,  
we have  
\begin{equation*}
\|\dbar u \|_{h, \omega}^{2} + 
\| \dbar^{*}_{h,\omega}u \|_{h, \omega}^{2} 
 = 
\| {D_{h}'}u  \|_{h, \omega}^{2} + \| {D_{h,\omega}'^{*}}u  \|_{h, \omega}^{2} + 
\lla \sqrt{-1}\Theta_{h}(G)\Lambda_{\omega} u, u
\rra_{h, \omega}. 
\end{equation*} 
\end{prop}

For the proof of our results, 
it is important to use special characteristics of 
canonical bundles (differential $(n,q)$-forms). 
By the following lemma, 
we can compare the norms of $(n,q)$-forms and $(p,0)$-forms  
with respect to different positive $(1,1)$-forms. 
Lemma \ref{nq} is obtained from straightforward computations, 
and thus we omit the proof.

\begin{lemm}\label{nq}
Let $\omega$ and $\ome$  be positive $(1,1)$-forms such that $\omega \leq \ome$. 
Then we have the following\,$:$
\begin{itemize}
\item[$\bullet$] There exists $C>0$ such that 
$|a \wedge b|_{\omega}\leq C |a|_{\omega}|b|_{\omega}$ 
for differential forms $a$, $b$.
\item[$\bullet$] The inequality $|a|^2_{\ome} \leq |a|^2_{\omega}$ holds for a differential form $a$. 
\item[$\bullet$] The inequality $|a|^2_{\ome}\, dV_{\ome} 
\leq |a|^2_{\omega}\,  dV_{\omega}$ holds
for a $(n,q)$-form $a$. 
\item[$\bullet$] The inequality $|a|^2_{\ome}\, dV_{\ome} 
\geq |a|^2_{\omega}\,  dV_{\omega}$ holds
for a $(p,0)$-form $a$. 
\item[$\bullet$] The equality $|a|^2_{\ome}\, dV_{\ome} 
= |a|^2_{\omega}\,  dV_{\omega}$ holds
for a $(n,0)$-form $a$. 
\end{itemize}
\end{lemm}

\subsection{De Rham-Weil isomorphisms} \label{Sec2-3}

In this subsection, 
we explain facts on the De Rham-Weil isomorphism 
from the $\dbar$-cohomology to the $\rm{\check{C}}$ech cohomology. 
The contents of this subsection are essentially contained in 
\cite{Fuj13a} and \cite{Mat16}, 
but we will summarize them for the reader's convenience.

Let $\omega$ be a K\"ahler form on a compact K\"ahler manifold $X$ 
and $h$ be a singular hermitian metric on a (holomorphic) line bundle $G$ 
such that $\sqrt{-1}\Theta_{h}(G) \geq - \omega$. 
Further let $Z$ be a proper subvariety on $X$ and 
let $\ome$ be a K\"ahler form on the Zariski open set $Y:=X \setminus Z$ 
with the following properties\,$:$
\begin{itemize}
\item[(B)] $\ome \geq \omega $ on $Y=X\setminus Z$. 
\item[(C)] For every point $p$ in $X$, 
there exists a \lq \lq bounded" function $\Phi$ 
on an open neighborhood of $p$ in $X$ 
such that $\ome =\deldel \Phi$. 
\end{itemize} 
The important point is that $\ome$ locally admits 
a \lq \lq bounded" potential function on a neighborhood of every point $p$ in $X$ (not $Y$), 
which enables us to construct the De Rham-Weil isomorphism.

As explained in subsection \ref{Sec2-2}, 
for the $L^2$-space of $G$-valued $(n,q)$-forms on $Y$ 
with respect to $h$ and $\ome$
\begin{align*}
L_{(2)}^{n, q}(G)_{h, \ome}:=
L_{(2)}^{n, q}(Y, G)_{h, \ome}:= 
\{u  \,|\,  u \text{ is a }G\text{-valued }(n, q)\text{-form with } 
\|u \|_{h, \ome}< \infty \}, 
\end{align*}
we have the orthogonal decomposition 
$$
L_{(2)}^{n, q}(G)_{h, \ome}=
\overline{\Image \dbar} \oplus 
\mathcal{H}^{n,q}_{h,\ome}(G) \oplus  
\overline{\Image \dbar^{*}_{h, \ome}}. 
$$
See subsection \ref{Sec2-2} for the set of harmonic forms 
$\mathcal{H}^{n,q}_{h,\ome}(G)$ with respect to $h$ and $\ome$. 
The following proposition is proved 
by the observation on the De Rham-Weil isomorphism 
(see \cite[Proposition 5.8]{Mat16} for the precise proof.)

\begin{prop}[{\cite[Proposition 5.8]{Mat16}}]\label{decomp}
Consider the same situation as above. 
That is, we consider a K\"ahler form $\omega$ on a compact K\"ahler manifold $X$, 
a singular hermitian metric $h$ on a $($holomorphic$)$ line bundle $G$ 
such that $\sqrt{-1}\Theta_{h}(G) \geq - \omega$, 
and a K\"ahler form $\ome$ on a Zariski open set $Y$ with properties $(B)$, $(C)$. 
Then the ranges ${\Image \dbar}$ and $\Image \dbar^{*}_{h, \ome}$ 
are closed subspaces in $L_{(2)}^{n, q}(G)_{h, \ome}$. 
In particular, we have the orthogonal decomposition 
$$
L_{(2)}^{n, q}(G)_{h, \ome}=
\Image \dbar \oplus 
\mathcal{H}^{n,q}_{h,\ome}(G) \oplus  
\Image \dbar^{*}_{h, \ome}. 
$$
\end{prop}

We fix a finite open cover 
$\mathcal{U}:=\{B_{i}\}_{i \in I}$ of $X$ by 
sufficiently small Stein open sets $B_{i}$. 
We consider the set of $q$-cochains 
$C^{q}(\mathcal{U}, K_{X}\otimes G \otimes \I{h})$ 
with coefficients in  $K_{X}\otimes G \otimes \I{h}$ calculated by $\mathcal{U}$ 
and the coboundary operator 
$$
\delta : C^{q}(\mathcal{U}, K_{X}\otimes G \otimes \I{h}) \to 
C^{q+1}(\mathcal{U}, K_{X}\otimes G \otimes \I{h}). 
$$
Then we have the isomorphism 
$$
\dfrac{{\rm{Ker}}\, \delta}{{\rm{Im}}\, \delta} 
\text{ of } C^{q}(\mathcal{U}, K_{X}\otimes G \otimes \I{h})
\cong  \check{H}^{q}(X, K_{X}\otimes G \otimes \I{h}) 
$$
since the open cover $\mathcal{U}$ is a Stein cover. 
By using suitable local solutions of the $\dbar$-equation, 
we can construct the De Rham-Weil isomorphism  
\begin{align*}
\overline{f_{h, \ome}} \colon \dfrac{{\rm{Ker}}\, \dbar}{{\rm{Im}}\, \dbar} 
\text{ of } L^{n,q}_{(2)}(G)_{h, \ome}  
\xrightarrow{\quad \cong \quad }
& \dfrac{{\rm{Ker}}\, \delta}{{\rm{Im}}\, \delta} 
\text{ of } C^{q}(\mathcal{U}, K_{X}\otimes G \otimes \I{h}). 
\end{align*}
Then, by the construction of $\overline{f_{h, \ome}}$ (see \cite[Proposition 5.5]{Mat16}), 
we can easily check the following proposition

\begin{prop}\label{comm}
Consider the same situation as in Proposition \ref{decomp}. \\
$(1)$ Then the following diagram is commutative\,$:$ 
\begin{align*}
\xymatrix{
&\check{H}^{q}(X, K_{X}\otimes G \otimes \I{h})
\ar@{=}[r] & \check{H}^{q}(X, K_{X}\otimes G \otimes \I{h})\\
&\ar[u]^-{ \cong}_-{\overline{f_{h, \omega}}}
\dfrac{\Ker \dbar}{\Image \dbar} \text{ of } 
L^{n,q}_{(2)}(G)_{h, \omega} \ar[r]^{j_{1}}
& \ar[u]^-{ \cong}_-{\overline{f_{h, \ome}}}
\dfrac{\Ker \dbar}{\Image \dbar} \text{ of } 
L^{n,q}_{(2)}(G)_{h, \ome}, \ar[u]}
\end{align*}
where $j_{1}$ is the map induced by the natural map 
$L^{n,q}_{(2)}(G)_{h, \omega} \to L^{n,q}_{(2)}(G)_{h, \ome}$. 
\\
$(2)$ Let $h'$ be a singular hermitian metric on $G$ such that 
$\sqrt{-1}\Theta_{h'}(G) \geq -\omega$ and $h' \geq h$. 
Then the following diagram is commutative\,$:$ 
\begin{align*}
\xymatrix{
&\check{H}^{q}(X, K_{X}\otimes G \otimes \I{h'})
\ar[r]^-{ j} & \check{H}^{q}(X, K_{X}\otimes G \otimes \I{h})\\
&\ar[u]^-{ \cong}_-{\overline{f_{h', \omega}}}
\dfrac{\Ker \dbar}{\Image \dbar} \text{ of } 
L^{n,q}_{(2)}(G)_{h', \ome} \ar[r]^{j_{2}}
& \ar[u]^-{ \cong}_-{\overline{f_{h, \ome}}}
\dfrac{\Ker \dbar}{\Image \dbar} \text{ of } 
L^{n,q}_{(2)}(G)_{h, \ome}, \ar[u]}\end{align*}
where $j_{2}$ is the map induced by the natural map 
$L^{n,q}_{(2)}(G)_{h', \ome} \to L^{n,q}_{(2)}(G)_{h, \ome}$ and 
$j$ is the map induced by $\I{h'} \hookrightarrow \I{h}$. 
\end{prop}

\begin{rem}\label{rem-comm}
By property (B) and the third claim of Lemma \ref{nq}, 
we have $\|u\|_{h,\ome} \leq \|u\|_{h,\omega}$ 
for an arbitrary $G$-valued $(n,q)$-form $u$. 
Therefore the natural map $j_{1}$ is well-defined. 
By the same way, 
we can easily check that $j_{2}$ is well-defined 
from $\|u\|_{h,\ome} \leq \|u\|_{h',\ome}$. 
\end{rem}

\subsection{Weak convergence in Hilbert spaces}\label{Sec2-4}
In this subsection, 
we summarize Lemma \ref{w-clo} and Lemma \ref{w-con}. 
See \cite[Section 2]{FM16} for the proof.

\begin{lemm}\label{w-clo}
Let $L$ be a closed subspace in a Hilbert space $\mathcal{H}$. 
Then $L$ is closed with respect to the weak topology of $\mathcal{H}$, 
that is, if a sequence $\{w_{k}\}_{k=1}^{\infty}$ in $L$ 
weakly converges to $w$, then the weak limit $w$ belongs to $L$. 
\end{lemm}

\begin{lemm}\label{w-con}
Let $\varphi: \mathcal{H}_{1} \to \mathcal{H}_{2}$ be 
a bounded operator $($continuous linear map$)$ between Hilbert spaces
$\mathcal{H}_{1}$ and $\mathcal{H}_{2}$. 
If $\{w_{k}\}_{k=1}^{\infty}$ weakly converges to $w$ in $\mathcal{H}_{1}$, 
then $\{\varphi(w_{k})\}_{k=1}^{\infty}$ weakly converges to 
$\varphi(w)$ in $\mathcal{H}_{2}$. 
\end{lemm}

\section{Proof of the main results}\label{Sec3}
\subsection{Proof of Theorem \ref{Lef}}\label{Sec3-1}
Theorem \ref{Lef} is a refinement of the hard Lefschetz theorem 
with multiplier ideal sheaves 
and plays a crucial role in the proof of Theorem \ref{main}. 
In this subsection, we give a proof of Theorem \ref{Lef}. 
We first show the following proposition, 
which is needed for the proof of Theorem \ref{Lef}.

\begin{prop}\label{orth}
Let $\omega$ be a K\"ahler form on a compact K\"ahler manifold $X$ and 
$(G, h)$ be a singular hermitian line bundle with semi-positive curvature. 
Let $\ome$ be a K\"ahler form on a non-empty Zariski open set $Y$ 
with the following properties\,$:$ 
\begin{itemize}
\item[(B)] $\ome \geq \omega$ on $Y$. 
\item[(C)] For every point $p \in X$, 
there exists a bounded function $\Phi$ on an open neighborhood of $p$ in $X$ 
such that $\ome = \deldel \Phi$. 
\end{itemize}
Then, we have $\lla u, w \rra_{h,\omega}=0 $
for any $u \in \mathcal{H}^{n,q}_{h,\omega}(G)$ 
and $w \in L^{n,q}_{(2)}(G)_{h,\omega}$ such that 
$w \in \Image \dbar \subset L^{n,q}_{(2)}(G)_{h,\ome}$. 
\end{prop}

\begin{rem}\label{rem-orth}
By the assumption $w \in \Image \dbar \subset L^{n,q}_{(2)}(G)_{h,\ome}$, 
there exists $v \in L^{n,q-1}_{(2)}(G)_{h,\ome}$ such that $w=\dbar v$. 
However, since the solution $v$ may not belong to $L^{n,q-1}_{(2)}(G)_{h,\omega}$, 
we can not immediately conclude that 
$w \in \Image \dbar \subset L^{n,q}_{(2)}(G)_{h,\omega}$. 
\end{rem}

\begin{proof}
Note that we have $w \in \Ker \dbar \subset L^{n,q}_{(2)}(G)_{h,\omega}$ 
by the assumption $w \in \Image \dbar \subset L^{n,q}_{(2)}(G)_{h,\ome}$. 
%It follows that $\dbar w = 0$ on $X$ from 
%$\dbar w=0$ on $Y$ and $\|w\|_{h,\omega}<\infty$ 
%(see 
%\cite[(7.3) Lemma, Chapter V\hspace{-.1em}I\hspace{-.1em}I\hspace{-.1em}I]{Dem-book}). 
By applying Proposition \ref{decomp} for $\omega$, 
we obtain the orthogonal decomposition
$$
L^{n,q}_{(2)}(G)_{h,\omega} \supset \Ker \dbar =
\Image \dbar \oplus 
\mathcal{H}^{n,q}_{h,\omega}(G). 
$$
By this orthogonal decomposition, $w$ can be decomposed as follows\,$:$
$$
w=w_{1}+w_{2} \ \text{ for some } 
w_{1} \in \Image \dbar \ \text{ and }\  
w_{2} \in \mathcal{H}^{n,q}_{h,\omega}(G) \ \text{ in } L^{n,q}_{(2)}(G)_{h,\omega}.
$$
We will show that $w_{2}$ is actually zero 
by the assumption $w \in \Image \dbar \subset L^{n,q}_{(2)}(G)_{h,\ome}$. 
Then we obtain the conclusion $\lla u, w \rra_{h,\omega}=0 $ 
since we have 
$$
\lla u, w \rra_{h,\omega} = \lla u, w_{2} \rra_{h,\omega}= 0 
\text{ by $u \in \mathcal{H}^{n,q}_{h,\omega}(G)$ and 
$w_{1} \in \Image \dbar \subset L^{n,q}_{(2)}(G)_{h,\omega}$.}
$$ 
To prove that $w_{2}=0$, 
we consider the following composite map\,:
$$
\phi : \mathcal{H}^{n,q}_{h,\omega}(G) \xrightarrow{} 
\frac{\Ker \dbar}{\Image \dbar} \text{ of }  L^{n,q}_{(2)}(G)_{h,\omega} \xrightarrow{j_{1}} 
\frac{\Ker \dbar}{\Image \dbar} \text{ of }  L^{n,q}_{(2)}(G)_{h,\ome},  
$$
where $j_{1}$ is the map induced by the natural map 
$L^{n,q}_{(2)}(G)_{h, \omega} \to L^{n,q}_{(2)}(G)_{h, \ome}$. 
The map $\phi $ is a (well-defined) isomorphism 
by Proposition \ref{decomp} and Proposition \ref{comm}. 
It follows that 
$$
w_{1} \in  \Image \dbar \subset L^{n,q}_{(2)}(G)_{h,\omega} \subset 
\Image \dbar \subset L^{n,q}_{(2)}(G)_{h,\ome}
$$ 
from the third claim of Lemma \ref{nq} and property (B) of $\ome$. 
Hence  
$w_{2}=w-w_{1}$ also belongs to 
$\Image \dbar \subset L^{n,q}_{(2)}(G)_{h,\ome}$ 
by the assumption $w \in \Image \dbar \subset L^{n,q}_{(2)}(G)_{h,\ome}$. 
In particular, this implies that  $\phi (w_{2})=0$.  
We obtain $w_{2}=0$ since the map $\phi $ is an isomorphism. 
\end{proof}

In the rest of this subsection, 
we prove Theorem \ref{Lef}.  

\begin{theo}[=Theorem \ref{Lef}]\label{Leff}
Let $\omega$ be a K\"ahler form on a compact K\"ahler manifold $X$ and 
$(G, h)$ be a singular hermitian line bundle with semi-positive curvature. 
Assume that the singular hermitian metric $h$ is 
smooth on a non-empty Zariski open set in $X$. 
Then, for a harmonic $G$-valued $(n,q)$-form $u \in \mathcal{H}^{n,q}_{h,\omega}(G)$ 
with respect to $h$ and $\omega$, 
we have 
$$
* u \in H^{0}(X, \Omega_{X}^{n-q}\otimes G \otimes \I{h}), 
$$
where $*$ is the Hodge star operator with respect to $\omega$.   
\end{theo}
\begin{proof}

Let $Y$ be a non-empty Zariski open set in $X$ such that $h$ is smooth on $Y$. 
We first take a complete K\"ahler form  $\ome$ on $Y$ 
with the following properties\,: 
\begin{itemize}
\item[$\bullet$] $\ome$ is a complete K\"ahler form on $Y$.
\item[$\bullet$] $\ome \geq \omega $ on $Y$. 
\item[$\bullet$] For every point $p \in X$, 
there exists a bounded function $\Phi$ on an open neighborhood of $p$ in $X$ 
such that $\omega = \deldel \Phi$.     
\end{itemize} 
See \cite[Section 3]{Fuj12a} for the construction of $\ome$. 
For the K\"ahler form $\omega_{\delta}$ on $Y$ defined by  
$$
\omega_{\delta}:=\omega + \delta \ome 
\, \text{ for }\, \delta>0, 
$$
it is easy to check the following properties\,: 
\begin{itemize}
\item[(A)] $\omega_{\delta}$ is a complete K\"ahler form on $Y$ for every  $\delta>0$.
\item[(B)] $\omega_{\delta_{2}} \geq  \omega_{\delta_{1}} \geq \omega $ on $Y$ 
for $\delta_{2} \geq \delta_{1} > 0$. 
\item[(C)] For every point $p \in X$, 
there exists a bounded function $\Phi_{\delta}$ on an open neighborhood of $p$ in $X$ 
such that $\omega_{\delta} = \deldel \Phi_{\delta}$.  
\end{itemize} 
Note that we can apply Proposition \ref{BKN} for $\omega_{\delta}$ thanks to property (A).  
In the proof of Theorem \ref{Lef}, 
we will omit the subscription $h$ of the norm, the $L^2$-space, and so on. 
For example, we will use the notations 
$$
\|\bullet \|_{\omega} := \|\bullet \|_{h, \omega}, \ 
\|\bullet \|_{\omega_{\delta}} := \|\bullet \|_{h, \omega_{\delta}}, \text{ and }
L^{n,q}_{(2)}(G)_{\omega_{\delta}}:=L^{n,q}_{(2)}(G)_{h,\omega_{\delta}}. 
$$
It follows that 
 \begin{align}\label{a}
\|u\|_{\omega_{\delta}} \leq \|u\|_{\omega} < \infty 
\end{align}
from Lemma \ref{nq} and property (B). 
In particular $u$ belongs to $L^{n,q}_{(2)}(G)_{\omega_{\delta}}$ 
for every $\delta>0$. 
By the orthogonal decomposition (see Proposition \ref{orth}) 
$$
L^{n,q}_{(2)}(G)_{\omega_{\delta}}= 
\Image \dbar \, \oplus 
\mathcal{H}^{n,q}_{\omega_{\delta}} (G) \, \oplus 
\Image \dbar_{\omega_{\delta}}^{*}, 
$$
the $G$-valued $(n,q)$-form $u$ can be decomposed as follows\,:
$$
u=w_{\delta}+u_{\delta} \ \text{ for some } 
w_{\delta} \in \Image \dbar \ \text{ and }\  
u_{\delta} \in \mathcal{H}^{n,q}_{\omega_{\delta}}(G) 
\ \text{ in } L^{n,q}_{(2)}(G)_{\omega_{\delta}}.
$$
The strategy of the proof is as follows\,: 
In the first step, 
we check that $u_\delta$ weakly converges to some $u_{0}$ in suitable $L^2$-spaces. 
In the second step, 
we show that the limit $u_{0}$ actually coincides with $u$ 
by Proposition \ref{orth}. 
In the third step, 
we prove that 
$*_{\delta} u_{\delta} \in H^{0}(X, \Omega_{X}^{n-q}\otimes G \otimes \I{h}) $ 
by the theory of harmonic integrals 
and $*_{\delta} u_{\delta}$ converges to $*u_{0}=*u$, 
where $*_{\delta}$ (resp. $*$) is the Hodge star operator 
with respect to $\omega_{\delta}$ (resp. $\omega$). 

We first check that $u_{\delta}$ has a a suitable weak limit 
by the following proposition. 
Since we use Cantor's diagonal argument in the proof of Proposition \ref{limit}, 
we need to handle only a countable sequence $\{\delta'\}_{\delta'>0}$. 

\begin{prop}\label{limit}
For a countable sequence $\{\delta'\}_{\delta'>0}$ converging to zero, 
there exist a subsequence $\{\delta_{\nu}\}_{\nu=1}^{\infty}$ of 
$\{\delta\}_{\delta>0}$ and $u_{0} \in L^{n,q}_{(2)}(G)_{\omega}$
with the following properties\,$:$
\begin{itemize}
\item[$\bullet$] For every  $\delta'>0$, 
as $\delta_{\nu}$ goes to $0$, 
$$
u_{\delta_{\nu}} \text{ converges to } u_{0} 
\text{ with respect to the weak } \text{topology in } 
L^{n,q}_{(2)}(G)_{\omega_{\delta'}}. 
$$ \item[$\bullet$] $\| u_{0} \|_{\omega}\leq \|u\|_{\omega}. $
\end{itemize}
\end{prop}

\begin{rem}\label{rem-limit} 
The subsequence $\{\delta_{\nu} \}_{\nu=1}^{\infty}$ and 
the weak limit $u_{0}$ do not depend on $\delta'$. 
The $G$-valued $(n,q)$-form $u_{\delta_{\nu}}$ weakly converges to $u_{0}$ 
in $L^{n,q}_{(2)}(G)_{\omega_{\delta'}}$, 
but we do not know whether $u_{\delta_{\nu}}$ weakly converges 
to $u_{0}$ in $L^{n,q}_{(2)}(G)_{\omega}$. 
\end{rem}

\begin{proof}
For a given $\delta'>0$, 
the sequence $\{u_{\delta}\}_{\delta' \geq \delta >0}$ 
is bounded in $L^{n,q}_{(2)}(G)_{\omega_{\delta'}}$. 
Indeed, for $\delta' \geq \delta >0$, 
we obtain 
\begin{align}\label{b}
\|u_{\delta}\|_{\omega_{\delta'}} \leq 
\|u_{\delta}\|_{\omega_{\delta}} \leq 
\|u\|_{\omega_{\delta}} \leq 
\|u\|_{\omega} < \infty. 
\end{align}
The first inequality follows 
from Lemma \ref{nq} and $\omega_{\delta'} \geq \omega_{\delta}$, 
the second inequality follows 
since $u_{\delta}$ is the orthogonal projection of $u$ in $L^{n,q}_{(2)}(G)_{\omega_{\delta}}$, 
and the third inequality follows from inequality (\ref{a}). 
Hence there exists a subsequence $\{\delta_{\nu}\}_{\nu=1}^{\infty}$ 
of $\{\delta\}_{\delta>0}$ such that 
$u_{\delta_\nu}$ weakly converges to some $u_{0, \delta'}$ in 
$L^{n,q}_{(2)}(G)_{\omega_{\delta'}}$, which may depend on $\delta'$. 
We can choose a suitable subsequence independent of $\delta'$ 
by Cantor's diagonal argument, 
and thus we can assume that 
this subsequence $\{\delta_{\nu}\}_{\nu=1}^{\infty}$ 
is independent of $\delta'$.

Now we show that the weak limit $u_{0, \delta'}$ is also independent of $\delta'$. 
For any $\delta_1 \geq \delta_2 $, 
the natural inclusion 
$L^{n,q}_{(2)}(G)_{\omega_{\delta_2}} 
\rightarrow L^{n,q}_{(2)}(G)_{\omega_{\delta_1}}$ 
is a bounded operator (continuous linear map) 
by Lemma \ref{nq} and $\omega_{\delta_1 } \geq \omega_{\delta_2}$. 
By Lemma \ref{w-con}, we can see that 
$u_{\delta_{\nu}}$ weakly converges to $u_{0, \delta_2}$ 
not only in $L^{n,q}_{(2)}(G)_{\omega_{\delta_2}}$ but also in 
$L^{n,q}_{(2)}(G)_{\omega_{\delta_1}}$. 
Therefore it follows that $u_{0, \delta_1}=u_{0, \delta_2}$ 
since $u_{\delta_{\nu}}$ weakly converges to $u_{0, \delta_1}$ in 
$L^{n,q}_{(2)}(G)_{\omega_{\delta_1}}$ and the weak limit is uniquely determined. 

Finally we estimate the $L^2$-norm of the weak limit $u_{0}$. 
Fatou's lemma yields 
$$
\|u_{0}\|^2_{\omega} = \int_{Y} |u_{0}|^2_{\omega}\, dV_{\omega}\leq
\liminf_{\delta' \to 0} 
\int_{Y} |u_{0}|^2_{\omega_{\delta'}}\, dV_{\omega_{\delta'}}=
\liminf_{\delta' \to 0} \|u_{0}\|^2_{\omega_{\delta'}}.  
$$
On the other hand, 
it is easy to see that 
$$
\|u_{0}\|_{\omega_{\delta'}} \leq 
\liminf_{\delta_{\nu} \to 0} 
\|u_{\delta_{\nu}}\|_{\omega_{\delta'}} \leq 
\liminf_{\delta_{\nu} \to 0} 
\|u_{\delta_{\nu}}\|_{\omega_{\delta_{\nu}}} \leq 
\|u\|_{\omega} < \infty.
$$
The first inequality follows since 
the norm is lower semi-continuous with respect to the weak convergence, 
the second inequality follows from Lemma \ref{nq} and 
$\omega_{\delta_{\nu}} \leq \omega_{\delta'} $, 
the third inequality follows from inequality (\ref{b}). 
These inequalities lead to 
the desired inequality $\| u_{0} \|_{\omega}\leq \|u\|_{\omega}$. 
\end{proof}

For simplicity, we use the same notation $\{u_{\delta}\}_{\delta>0}$ 
for the subsequence $\{u_{\delta_{\nu}}\}_{\delta_{\nu}>0}$ 
chosen in Proposition \ref{limit}. 
The following proposition is obtained from Proposition \ref{orth}.

\begin{prop}\label{limit2}
The weak limit $u_{0}$ coincides with $u$. 
\end{prop}

\begin{proof}
We fix $\delta_{0}>0$ in the proof of Proposition \ref{limit2}. 
By Lemma \ref{nq}, we can see that  
$$
\Image \dbar \text{ in } L^{n,q}_{(2)}(G)_{\omega_{\delta}} 
\subset 
\Image \dbar \text{ in } L^{n,q}_{(2)}(G)_{\omega_{\delta_{0}}}
$$
for an arbitrary $\delta$ with $\delta_{0} \geq \delta >0$. 
Hence, it follows that 
$$
u-u_{\delta}=w_{\delta}
\in \Image \dbar \text{ in } L^{n,q}_{(2)}(G)_{\omega_{\delta_{0}}}  
$$
from the construction of $u_{\delta}$ and $w_{\delta}$. 
The subspace $\Image \dbar$ is closed not only with respect to 
the $L^2$-topology but also with respect to the weak topology 
(see Proposition \ref{decomp} and Lemma \ref{w-clo}). 
By taking the weak limit, we can conclude that 
$$
w_{0}:=u-u_{0}= \text{\rm{w}-}\lim_{\hspace{-0.6cm}\delta \to 0}w_{\delta} 
\in \Image \dbar \text{ in } L^{n,q}_{(2)}(G)_{\omega_{\delta_{0}}}. 
$$
Since the K\"ahler form $\omega_{\delta_{0}}$ on $Y$ satisfies 
properties (B) and (C), 
we have $\lla u, w_{0} \rra_{\omega}=0$ by Proposition \ref{orth}, 
where $w_{0}$ is the weak limit of $w_{\delta}=u-u_{\delta}$. 
Hence we obtain $\|u_{0}\|^2_{\omega}=\|u\|^2_{\omega}+\|w_{0}\|^2_{\omega}$. 
This is a contradiction to 
the inequality $\|u_{0} \|_{\omega} \leq \|u \|_{\omega}$ in Proposition \ref{limit}
if $w_{0}$ is not zero. 
Therefore $w_{0}$ is actually zero. 
We obtain the desired conclusion $u=u_{0}$. 
\end{proof}

From now on, we consider the Hodge star operator $*_{\delta}$ with respect to $\omega_{\delta}$ 
and the $G$-valued $(n-q,0)$-form $*_{\delta}u_{\delta}$. 
Note that $*_{\delta}u_{\delta}$ is a $G$-valued $(n-q, 0)$-form 
on $Y$ (not $X$)  
since the K\"ahler form $\omega_{\delta}$ is defined only on $Y$. 
However, by the following proposition, 
we can regard $*_{\delta}u_{\delta}$ 
as a holomorphic $G$-valued $(n-q,0)$-form on $X$.

\begin{prop}\label{hol}
The $G$-valued $(n-q, 0)$-form $*_{\delta}u_{\delta}$ 
can be extended to a holomorphic $G$-valued $(n-q,0)$-form on $X$ 
$($that is, $\dbar *_{\delta}u_{\delta} =0$ on $X$$)$.
Moreover we have 
$$
\|*_{\delta}u_{\delta} \|_{\omega} \leq \|u\|_{\omega} < \infty. 
$$ 
In particular, we have 
$$*_{\delta}u_{\delta} \in 
H^{0}(X, \Omega_{X}^{n-q}\otimes G \otimes \I{h}). 
$$
\end{prop}
\begin{proof}
Let $*_{\delta} u_{\delta}=\sum_{J} f_{J} dz_{J}$ 
be a local expression in terms of a local coordinate 
$(z_{1}, z_{2}, \dots, z_{n})$, where $J$ is 
an ordered multi-index with degree $(n-q)$. 
We will show that every coefficient $f_{J}$ is holomorphic on $Y$ and 
can be extended to a holomorphic function on $X$.

Since $\omega_{\delta}$ is a complete K\"ahler form on $Y$, 
we can apply Proposition \ref{BKN} to $u_{\delta}$. 
Proposition \ref{BKN} yields 
\begin{align}\label{c}
0=\|\dbar u_{\delta} \|^{2}_{\omega_{\delta}}+
\|\dbar_{\delta}^{*} u_{\delta} \|^{2}_{\omega_{\delta}}=
\|D'^*_{\delta} u_{\delta} \|^{2}_{\omega_{\delta}} + 
\lla \sqrt{-1}\Theta_{h}(G)\Lambda_{\omega_{\delta}}u_{\delta},  
u_{\delta} \rra_{\omega_{\delta}}. 
\end{align}
The first equality follows 
since $u_{\delta}$ is harmonic with respect to $\omega_{\delta}$. 
Here $D'^*_{\delta}$ denotes the Hilbert space adjoint of 
the $(1,0)$-part of the Chern connection $D_{h}=D'_{h}+\dbar$ 
and $\Lambda_{\omega_{\delta}}$ denotes the adjoint operator of the wedge product 
$\omega_{\delta} \wedge \bullet$.

The second term of the right hand side is non-negative 
by the assumption $\sqrt{-1}\Theta_{h}(G) \geq 0$, 
and thus the first term and the second term must be zero. 
In particular we obtain $|D'^*_{\delta} u_{\delta}|_{\omega_{\delta}}=0$
by
$$
0=\|D'^*_{\delta} u_{\delta} \|^{2}_{\omega_{\delta}} = 
\int_{Y} 
|D'^*_{\delta} u_{\delta} |^{2}_{\omega_{\delta}}\, dV_{\omega_{\delta}}. 
$$
The Hilbert space adjoint coincides with the formal adjoint 
since $\omega_{\delta}$ is complete 
(see, for example, \cite[(3.2) Theorem 
in Chapter  V\hspace{-.1em}I\hspace{-.1em}I\hspace{-.1em}I]{Dem-book}). 
Hence we have $D'^*_{\delta}=-*_{\delta}\dbar *_{\delta}$. 
It follows that 
$$
0\equiv |D'^*_{\delta} u_{\delta} |_{\omega_{\delta}}=
|-*_{\delta}\dbar *_{\delta} u_{\delta}|_{\omega_{\delta}}=
|\dbar *_{\delta} u_{\delta}|_{\omega_{\delta}}
$$
since the Hodge star operator $*_{\delta}$ preserves 
the point-wise norm $|\bullet|_{\omega_{\delta}}$. 
Therefore the $G$-valued $(n-q,0)$-form $*_{\delta} u_{\delta}$ is $\dbar$-closed on $Y$, 
that is, the coefficient $f_{J}$ is a holomorphic function on $Y$.

Now we show that the $L^2$-norm of the coefficient $f_{J}$ 
with respect to $h$ is uniformly bounded     
(that is, $\int |f_{J}|^{2}_{h}\, dV_{\omega} < C$ for some $C>0$).  
The key point here is the following inequality\,:
\begin{align}\label{d}
\|*_{\delta} u_{\delta}\|_{\omega} \leq 
\|*_{\delta} u_{\delta}\|_{\omega_{\delta}} = 
\|u_{\delta}\|_{\omega_{\delta}} \leq 
\|u\|_{\omega}<\infty. 
\end{align}
The first inequality follows from 
the fourth claim of Lemma \ref{nq} and $\omega \leq \omega_{\delta}$, 
the second inequality follows since $*_{\delta}$ preserves 
the point-wise norm $|\bullet|_{\omega_{\delta}}$, 
the third inequality follows from inequality (\ref{b}). 
On the other hand, 
there is a constant $C'$ (independent of $\delta$) 
such that $|f_{J}|^2_h \leq C' |*_{\delta} u_{\delta}|^2_{\omega}$. 
Indeed, by  the first claim of Lemma \ref{nq}, 
we can easily check that 
\begin{align*}
|f_{J} |_{h} \inf 
(|dz_{J} \wedge dz_{\hat{J}} \wedge d\overline{z}|_{\omega}) &\leq 
|f_{J} dz_{J}\wedge dz_{\hat{J}} \wedge d\overline{z}|_{h,\omega} \\ &= 
|*_{\delta} u_{\delta} \wedge dz_{\hat{J}} \wedge d\overline{z}|_{h,\omega} \\ &\leq
C''|*_{\delta} u_{\delta}|_{\omega} \sup (|dz_{\hat{J}} \wedge d\overline{z}|_{\omega}),  
\end{align*}
for some positive constant $C''$ (independent of $\delta$), 
where $\hat{J}$ is the complementary index of $J$. 
By combining with inequality (\ref{d}), 
we obtain 
$$
\int |f_{J}|^{2}_{h}\, dV_{\omega} \leq 
C'\|*_{\delta} u_{\delta}\|^2_{\omega} \leq  
C'\|u\|^2_{\omega}
$$
Therefore, by the Riemann extension theorem, 
the coefficient $f_{J}$ can be extended to a holomorphic function on $X$. 
\end{proof}

We put $f_{\delta}:=*_{\delta}u_{\delta}$ and 
consider a local expression 
$f_{\delta} = *_{\delta} u_{\delta}=\sum_{J} f_{\delta, J} dz_{J}$ again. 
By the proof of Proposition \ref{hol}, we can see that 
the $L^2$-norm of the coefficient $f_{\delta, J}$ 
is uniformly bounded with respect to $\delta$. 
Hence, by Montel's theorem, there exists a subsequence $\{\delta_{\nu}\}_{\nu=1}^{\infty}$ 
of $\{\delta\}_{\delta>0}$ such that 
$f_{\delta_{\nu}}=*_{\delta_{\nu}}u_{\delta_{\nu}}$ uniformly converges to some $f_{0}$,  
that is,  the local sup-norm 
$\sup |f_{\delta_{\nu}, J} - f_{0, J}|$ converges to zero, 
where $f_{0, J}$ is the coefficient of $f_{0}=\sum_{J} f_{0, J} dz_{J}$. 
Then the $L^2$-norm $\|f_{\delta_{\nu}} - f_{0} \|_{h,\omega}$ 
also converges to zero (for example see \cite[Lemma 5.2]{Mat16}). 
In particular, the limit $f_{0}$ satisfies 
$f_{0} \in H^{0}(X, \Omega_{X}^{n-q}\otimes G \otimes \I{h})$. 
For simplicity we use the same notation $\{f_{\delta}\}_{\delta>0}$ for this subsequence. 
Then we show that $u_{0}$
(which is the weak limit obtained in Proposition \ref{limit}) 
coincides with $(-1)^{n+q} *f_{0}$. 

\begin{prop}\label{agree}
The weak limit $u_{0}$ coincides with $(-1)^{n+q}* f_{0}$. 
In particular, we can see that $u= (-1)^{n+q}*f_{0}$ by Proposition \ref{limit2}. 
\end{prop}
\begin{proof}
For a contradiction, 
we assume that $u_{0} \not = (-1)^{n+q}* f_{0}$ 
in $L^{n,q}_{(2)}(G)_{\omega_{\delta'}}$. 
Since the smooth $G$-valued $(n,q)$-forms with compact support in $Y$ 
is dense in $L^{n,q}_{(2)}(G)_{\omega_{\delta'}}$, 
there exists a smooth $G$-valued $(n,q)$-form $\eta$ with compact support in $Y$ 
such that 
$\lla u_{0}, \eta \rra_{\omega_{\delta'}}\not= 
\lla (-1)^{n+q}* f_{0}, \eta \rra_{\omega_{\delta'}}$. 
Since $u_{\delta}$ weakly converges to $u_{0}$ in 
$L^{n,q}_{(2)}(G)_{\omega_{\delta'}}$, 
we have 
$\lla u_{0}, \eta \rra_{\omega_{\delta'}}= 
\lim_{\delta \to 0}\lla u_{\delta}, \eta \rra_{\omega_{\delta'}}$. 
On the other hand, 
it follows that $*_{\delta}f_{\delta}$ uniformly converges $*f_{0}$ 
on every relatively compact set in $Y$
since $f_{\delta}$ uniformly converges $f_{0}$ and 
$\omega_{\delta}$ uniformly converges $\omega$ 
on every relative compact set in $Y$. 
Indeed, it is sufficient to consider $(*_{\delta}f_{\delta}- *f_{\delta})$ 
since we have   
\begin{align*}
&*_{\delta}f_{\delta}-*f_{0}=
(*_{\delta}f_{\delta}- *f_{\delta}) + (*f_{\delta} -*f_{0}),\\
&\sup_{X}|*f_{\delta} -*f_{0}|_{\omega}=\sup_{X}|f_{\delta} -f_{0}|_{\omega} 
\to 0. 
\end{align*}
For a relatively compact set $K$ in $Y$ and a given point $x \in K$, 
we take a local coordinate $(z_{1}, z_{2}, \dots, z_{n})$ 
centered at $x \in K$ such that 
$$
\omega=\frac{\sqrt{-1}}{2}\sum_{i=1}^{n} dz_{i} \wedge  d\overline{z_{i}} 
\text{ \quad and  \quad } 
\ome=\frac{\sqrt{-1}}{2}\sum_{i=1}^{n} \lambda_{i} dz_{i} \wedge  d\overline{z_{i}} \text{\quad at } x.
$$
By $K \Subset Y$, there exists a positive constant $C$ such that 
$0 \leq \ome \leq C \omega$ on $K$.
In particular we have $0 \leq \lambda_{i} \leq C$. 
Note that the eigenvalues of $\omega_{\delta}$ with respect to $\omega$ 
are $\{(1+\delta \lambda_{i})\}_{i=1}^{n}$. 
When $f_{\delta}$ is locally written as $f_{\delta}=\sum_{J} f_{\delta,J} dz_{J}$, 
we can easily see that 
\begin{align*}
|*_{\delta}f_{\delta}- *f_{\delta}|_{\omega} 
&=|\sum_{J} f_{\delta,J} (*_{\delta} dz_{J} -*dz_{J}) |_{\omega}
\\
&=|\sum_{J} f_{\delta,J}\   {\rm{sign}}(J\hat{J})\  \big\{\Pi_{i \in \hat{J}}
(1+\delta \lambda_{i}) -1 \big\}
dz_{(1,2,\dots,n)} \wedge d\overline{z_{\hat{J}}}|_{\omega}\\
&\leq \delta C' \sum_{J} \sup_{K}|f_{\delta,J}| |dz_{(1,2,\dots,n)} \wedge d\overline{z_{\hat{J}}}|_{\omega}
\end{align*}
for some constant $C'$. 
The coefficient $f_{\delta,J}$ is a holomorphic function, 
and thus the (local) sup-norm $\sup_{K}|f_{\delta,J}|$
of  $f_{\delta,J}$ can be bounded by 
the $L^2$-norm. 
Further  the $L^2$-norm of $f_{\delta,J}$ is uniformly bounded with respect to $\delta$ 
(see Proposition \ref{hol}). 
Therefore $(*_{\delta}f_{\delta}- *f_{\delta})$ uniformly 
converges to zero on $K \Subset Y$. 
Hence, by the definition of $f_{\delta} = *_{\delta} u_{\delta}$, 
we obtain 
\begin{align*}
\lla (-1)^{n+q}* f_{0}, \eta \rra_{\omega_{\delta'}}&= 
\lim_{\delta \to 0}\lla (-1)^{n+q}*_{\delta} f_{\delta}, 
\eta \rra_{\omega_{\delta'}}\\ &=
\lim_{\delta \to 0}\lla (-1)^{n+q}*_{\delta}*_{\delta} u_{\delta}, 
\eta \rra_{\omega_{\delta'}}\\
&=\lim_{\delta \to 0}\lla  u_{\delta}, 
\eta \rra_{\omega_{\delta'}}. 
\end{align*}
This is a contradiction to $\lla u_{0}, \eta \rra_{\omega_{\delta'}}\not= 
\lla (-1)^{n+q}* f_{0}, \eta \rra_{\omega_{\delta'}}$. 
Therefore we can conclude that $u_{0} = (-1)^{n+q}* f_{0}$ 
in $L^{n,q}_{(2)}(G)_{\omega_{\delta'}}$ for every $\delta'>0$. 
Then, by Fatou's lemma, we can easily see that 
\begin{align*}
\| u_{0} - (-1)^{n+q}* f_{0} \|_{\omega} \leq 
\liminf_{\delta' \to 0} \| u_{0} - (-1)^{n+q}* f_{0} \|_{\omega_{\delta'}}=0. 
\end{align*}
\end{proof}
By $f_{0} \in H^{0}(X, \Omega_{X}^{n-q}\otimes G \otimes \I{h})$, 
we obtain the desire conclusion
$$
* u =  (-1)^{n+q}**f_{0}=f_{0} \in H^{0}(X, \Omega_{X}^{n-q}\otimes G \otimes \I{h}) 
$$
in Proposition \ref{Lef}. 
This completes the proof. 
\end{proof}

\subsection{Proof of Theorem \ref{key}}\label{Sec3-2}
In this subsection, we prove Theorem \ref{key}.

\begin{theo}[=Theorem \ref{key}]\label{keyy}
Let $D$ be a simple normal crossing divisor on a compact K\"ahler manifold $X$. 
Let $F$ $($resp. $M$$)$ be a $($holomorphic$)$ line bundle on $X$ 
with a smooth hermitian metric $h_{F}$  $($resp. $h_{M}$$)$ 
such that 
\begin{equation*}
\sqrt{-1}\Theta_{h_M}(M)\geq 0 \  \text{ and }\ 
\sqrt{-1}(\Theta_{h_F}(F)-t \Theta 
_{h_M}(M))\geq 0 
\ \text{ for some }t>0. 
\end{equation*}
We consider the map 
\begin{equation*}
\Phi_{D} : 
H^{q}(X, K_{X} \otimes F) 
\xrightarrow{\quad} 
H^{q}(X, K_{X} \otimes D \otimes F )
\end{equation*}
induced by the natural inclusion $\OX_{X} \hookrightarrow \OX_{X}(D)$. 
Then, the multiplication map on the image $\Image \Phi_{D}$ 
induced by the tensor product with $s$ 
\begin{equation*}
\Image \Phi_{D}
\xrightarrow{\quad \otimes s \quad } 
H^{q}(X, K_{X} \otimes D \otimes F\otimes M )
\end{equation*}
is injective for every $q$.\end{theo}
\begin{proof}
Let $g$ be a smooth hermitian metric on the line bundle $D$ and 
$t$ be the natural section of the effective divisor $D$. 
Then we define the smooth hermitian metric $g_{\e}$ on the line bundle $D$ by 
$$
\varphi_{\e}:=\frac{1}{2} \log(|t|^2_{g} + \e) \text{\quad and \quad} 
g_{\e}:=g e^{-2 \varphi_{\e}}=g \cdot \Big(\frac{1}{|t|^2_{g} + \e}\Big). 
$$
It is easy to see that 
\begin{itemize}
\item $g_{\e_{2}} \leq g_{\e_{1}}$ for $\e_{1}\leq \e_{2}$, 
\item $g_{\e}$ converges to $g_{0}=h_{D}$ in the point-wise sense 
as $\e$ tends to zero, 
\end{itemize}
where $h_{D}$ is the singular metric defined by the effective divisor $D$ 
(see Example \ref{multi-ex}). 
We have $\I{g_{0}}=\I{h_{D}}=\mathcal{O}_{X}(-D)$ since $D$ is a simple normal crossing divisor. 
Let $\omega$ be a K\"ahler form on $X$, 
and let $h_{F}$ and $h_{M}$ be smooth hermitian metrics satisfying 
the assumptions in Theorem \ref{key}. 
We often omit the subscripts $\omega$, $h_{F}$, and $h_{M}$ 
of the norm, the $L^2$-space, and so on. 
For example, we use the notations 
$$
L^{n,q}_{(2)}(D\otimes F)_{g_{\e}}:=L^{n,q}_{(2)}(D\otimes F)_{g_{\e}h_{F}, \omega} 
\text{ and }
\mathcal{H}^{n,q}_{g_{0}}(D\otimes F):=
\mathcal{H}^{n,q}_{g_{0}h_{F},\omega}(D\otimes F). 
$$

We first consider the following commutative diagram\,: 
\begin{align*}
\xymatrix{
&H^{q}(X, K_{X}\otimes D \otimes F \otimes \I{g_{0}})=H^{q}(X, K_{X}\otimes F)\ar[r]^-{\Phi_{D}}&H^{q}(X, K_{X}\otimes D \otimes F)\\
&\ar[u]^-{ \cong}_-{\overline{f_{0}}}
\dfrac{\Ker \dbar}{\Image \dbar} \text{ of } 
L^{n,q}_{(2)}(D\otimes F)_{g_{0}}
& \ar[u]^-{ \cong}_-{\overline{f_{\e}}}
\dfrac{\Ker \dbar}{\Image \dbar} \text{ of } 
L^{n,q}_{(2)}(D\otimes F)_{g_{\e}}\ar[u]\\
&\ar[u]^-{ \cong}_-{j}\ar[ur]_-{\phi } \mathcal{H}^{n,q}_{g_{0}}(D\otimes F).&  
}
\end{align*}
Here $\overline{f_{0}}$ and $\overline{f_{\e}}$ 
are the De Rham-Weil isomorphisms given in subsection \ref{Sec2-3} 
and $j$ (resp. $\phi$) is the map induced by the natural inclusion 
$\mathcal{H}^{n,q}_{g_{0}}(D\otimes F) \hookrightarrow \Ker \dbar \subset 
L^{n,q}_{(2)}(D\otimes F)_{g_{0}}$ (resp. 
$\mathcal{H}^{n,q}_{g_{0}}(D\otimes F) \hookrightarrow \Ker \dbar \subset 
L^{n,q}_{(2)}(D\otimes F)_{g_{\e}}$). 
For a cohomology class $\alpha$ such that 
$\alpha \in \Image \Phi_{D} \subset H^{q}(X, K_{X}\otimes D \otimes F)$, 
we assume that $s \alpha = 0\in H^{q}(X, K_{X}\otimes D \otimes F \otimes M)$. 
Our goal is to show that the cohomology class $\alpha$ is actually zero 
under this assumption. 
By $\alpha \in \Image \Phi_{D}$, 
there exists a cohomology class 
$\beta \in H^{q}(X, K_{X} \otimes F)$ 
such that $\Phi_{D}(\beta)=\alpha$. 
By the above isomorphisms, 
the cohomology class $\beta$ can be represented by 
the harmonic form $u_{1} \in \mathcal{H}^{n,q}_{g_{0}}(D\otimes F)$ 
(that is, $\beta=\{ u_{1}\}$). 
Since $\mathcal{H}^{n,q}_{g_{0}}(D\otimes F)$ is a finite dimensional 
vector space with the inner product 
$\lla \bullet, \bullet \rra_{g_{0}}:=\lla \bullet, \bullet  \rra_{g_{0}h_{F},\omega}$, 
we have the orthogonal decomposition 
\begin{align}\label{g}
\mathcal{H}^{n,q}_{g_{0}}(D\otimes F)=
\Ker \phi \oplus (\Ker \phi)^{\perp}. 
\end{align}
From this orthogonal decomposition, the harmonic form $u_{1}$ can be decomposed as follows\,: 
$$
u_{1}=u_{2} + u \text{ for some } u_{2} \in \Ker \phi 
\text{ and } u \in (\Ker \phi)^{\perp}.  
$$ 
Then it is easy to see that 
$$
\Phi_{D}(\{u\})=\Phi_{D}(\{u_{2}+u\})=\Phi_{D}(\beta)=\alpha. 
$$
Note that $\{u_{2}+u\}$ is equal to  $\beta$, 
but it is not necessarily equal to $\{u\}$. 
We can see that if we can prove $u = 0$, 
we obtain $\alpha=0$ (the desired conclusion of Theorem \ref{key}). 
Hence our goal is to show $u = 0$. 

By the assumption $\sqrt{-1}\Theta_{h_{F}}(F)\geq 0$, 
the line bundle $G=D\otimes F$ and the singular hermitian metric $h=g_{0}h_{F}$ satisfy 
the assumptions in Theorem \ref{Lef}. 
By applying  Theorem \ref{Lef} for $u$, 
we obtain 
\begin{align}\label{f}
\ast u \in H^{0}(X, \Omega_{X}^{n-q}\otimes D\otimes F \otimes \I{g_{0}}). 
\end{align}
In particular $\ast u $ is smooth on $X$. 
Although $u$ is a priori $D \otimes F$-valued $(n,q)$-form on 
$Y:=X \setminus \Supp D$ (not $X$), 
it follows that $u=(-1)^{n+q}**u$ is smooth on $X$
from (\ref{f}).

\begin{rem}\label{Lef-rem}
(1) It seems to be difficult to show that $u $ is smooth on $X$ 
without using Theorem \ref{Lef}, 
since $g_{0}$ is a singular hermitian metric and $\omega$ is not complete on $Y$.\\
(2) Note that we have $\I{g_{0}}=\mathcal{O}(-D)$ since 
$D$ is a simple normal crossing divisor. 
Therefore $\ast u/t$ is a holomorphic  $F$-valued $(n-q,0)$-form. 
In particular $\ast u/t$ is still smooth on $X$, 
which plays a crucial role later. 
\end{rem}

By the standard De Rham-Weil isomorphism, 
we have 
$$
\Phi_{D} (\{su\})=s\alpha =0 \in H^{q}(X, K_{X}\otimes D \otimes F \otimes M) \cong 
\frac{\Ker \dbar}{\Image \dbar} \text{ of } C^{n,q}_{\infty} 
(D \otimes F \otimes M),  
$$
where $C^{n,q}_{\infty}(D \otimes F \otimes M)$ is the 
set of smooth $D \otimes F \otimes M$-valued $(n,q)$-forms on $X$.
Hence, by the assumption $s\alpha =0$,  
we can take a smooth $D \otimes F \otimes M$-valued $(n,q-1)$-form $v$ 
such that $su=\dbar v$. 
The bounded Lebesgue convergence theorem yields 
$$
\|su \|^2_{g_{0}}= \lim_{\e \to 0}\int_{Y} 
|su |^2_{g_{\e}}\, dV_{\omega} 
= \lim_{\e \to 0}  \lla su, su \rra_{g_{\e}}
$$ 
since $|su |^2_{g_{\e}} \leq |su |^2_{g_{0}}$ 
and $|su |^2_{g_{0}}$ is integrable. 
Therefore, from Cauchy-Schwartz's inequality, 
we obtain 
\begin{align}\label{i}
\|su \|^2_{g_{0}}= \lim_{\e \to 0}  \lla su, su \rra_{g_{\e}}
= \lim_{\e \to 0}  \lla su, \dbar v \rra_{g_{\e}}
\leq \lim_{\e \to 0}  \|\dbar_{g_{\e}}^{*} su \|_{g_{\e}} \| v\|_{g_{\e}}. 
\end{align}
%We remark that 
%$su$ belongs to the domain of 
%the closed operator $\dbar_{g_{\e}}^{*}$ 
%since $g_{\e}$ is smooth on $X$ and $X$ is compact. 

The strategy of the proof of Theorem \ref{key} is as follows\,: 
We will show that $\| v\|_{g_{\e}}=O(- \log \e)$ and 
$\|\dbar_{g_{\e}}^{*} su \|_{g_{\e}} = O(\e (-\log \e))$. 
Then, from inequality (\ref{i}), 
we obtain $\|su \|^2_{g_{0}}=0$ (that is, $su=0$). 
This completes the proof.

We first check the following lemma. 
\begin{lemm}\label{comp}
Let $(z_{1},z_{2},\dots, z_{n})$ be the standard coordinate of 
$\mathbb{C}^{n}$ and $B$ be an open ball containing the origin. 
Then, for every $1 \leq k \leq n$, we have 
$$
\int_{B} \frac{1}{\e + |z_{1}z_{2}\cdots z_{k}|^2}= 
O(-\log \e). 
$$ 
\end{lemm}
\begin{proof}
By the variable change $z_{i}=r_{i}e^{\sqrt{-1}\theta_{i}}$, 
the problem can be reduced to showing
$$
\int_{0 \leq r_{1} \leq 1} \int_{0 \leq r_{2} \leq 1}\dots 
\int_{0 \leq r_{k} \leq 1}
\frac{r_{1}r_{2}\cdots r_{k} }{\e + |r_{1}r_{2}\cdots r_{k}|^2}
dr_{1} dr_{2}\cdots dr_{k} = O(-\log \e). 
$$
Further, by using the polar coordinate, 
we can obtain the conclusion from the following computation\,:
\begin{align*}
\int_{0 \leq R \leq 1} 
\frac{R^{2k-1} }{\e + R^{2k}}dR 
= \frac{1}{2k}(\log(\e+1)-\log \e). 
\end{align*}
\end{proof}

By Lemma \ref{comp}, we can easily obtain the following proposition. 
In the proof of the following proposition, 
we essentially use the fact that $v$ is smooth on $X$. 

\begin{prop}\label{bdd}
$\| v\|_{g_{\e}}=O(- \log \e)$. 
\end{prop}
\begin{proof}
By the definition of $g_{\e}$, 
we can see that 
$$
\| v\|^2_{g_{\e}}=\int_{X}|v|^2_{g} \frac{1}{\e + |t|^2_{g}} dV_{\omega} 
\leq \sup_{X}|v|^2_{g} \int_{X} \frac{1}{\e + |t|^2_{g}} dV_{\omega}.  
$$
It follows that  $\sup_{X}|v|^2_{g}$ is finite since $v$ and $g$ 
are smooth on $X$. 
Since $D={\rm{div}}\,t$ is a simple normal crossing divisor, 
we can obtain the conclusion by Lemma \ref{comp}. 
\end{proof}

It remains to show that 
$$
\|\dbar_{g_{\e}}^{*} su \|_{g_{\e}} = O(\e (-\log \e)).
$$
By applying Proposition \ref{BKN} for $su$, $g_{\e}$, and $\omega$, 
we obtain  
\begin{align}\label{e}
\|\dbar_{g_{\e}}^{*} su \|^2_{g_{\e}}= 
\|D_{g_{\e}}'^* su \|^2_{g_{\e}} + 
\lla \sqrt{-1}\Theta_{g_{\e}h_{F}h_{M}}(D\otimes F \otimes M) 
\Lambda su, su \rra_{g_{\e}},  
\end{align}
where $D_{g_{\e}}'^*$ (resp. $\dbar_{g_{\e}}^{*}$) is 
the Hilbert space adjoint of the $(1,0)$-part $D_{g_{\e}}'$ 
(resp. the $(0,1)$-part $\dbar$)
of the Chern connection $D_{g_{\e}}=D_{g_{\e}}'+\dbar$, 
and $\Lambda$ is the adjoint operator of the wedge product $\omega \wedge \bullet$. 
Here we used that $\dbar su=s \dbar u=0$.

We consider the first term $\|D_{g_{\e}}'^* su \|_{g_{\e}}$ of the right hand side of (\ref{e}).  
It follows that $D_{g_{\e}}'^*=-*\dbar * $ since $X$ is compact and $\omega$ is defined on $X$. 
We have $\dbar * u=0$ by (\ref{f}) (see Theorem \ref{Lef}), 
and thus we obtain 
\begin{align}\label{j}
D_{g_{\e}}'^* su = 
-*\dbar * su =
-*\dbar s* u =
-*s \dbar * u =0.  
\end{align}
In particular we can see $\|D_{g_{\e}}'^* su \|_{g_{\e}}=0$. 

The problem is the second term of the right hand side of (\ref{e}). 
We can obtain  
\begin{align*}
\sqrt{-1}\Theta_{g_{\e}}(D)&= 
\e \frac{1} {|t|^2_{g}+\e}\sqrt{-1}\Theta_{g}(D)+
\e  \frac{D'_{g}t \wedge \overline{D'_{g}t} }{(|t|^2_{g}+\e)^2}, 
\end{align*}
where $D'_{g}$ is the $(1,0)$-part of the Chern connection $D_{g}$. 
The above equality follows from simple computations, 
but we will check it for reader's convenience. 
By the definition of $g_{\e}$ and the Chern connection, we have 
\begin{align*}
&\sqrt{-1}\Theta_{g_{\e}}(D)\\
=&\sqrt{-1}\Theta_{g}(D)+\deldel \log (|t|^2_{g}+\e)\\
=&\sqrt{-1}\Theta_{g}(D)+ 
\frac{-\sqrt{-1}}{(|t|^2_{g}+\e)^2}
\langle D'_{g}t, t \rangle_g \wedge \langle t, D'_{g}t \rangle_g + 
\frac{\sqrt{-1}}{|t|^2_{g}+\e}\bigg(
\langle D'_{g}t, D'_{g}t \rangle_g + 
\langle t, \dbar D'_{g}t \rangle_g
\bigg). 
\end{align*}
Further, by easy computations, 
we have  
\begin{align*}
\langle D'_{g}t, t \rangle_g \wedge \langle t, D'_{g}t \rangle_g 
= \langle D'_{g}t, D'_{g}t \rangle_g |t|_g^2 \quad \text{ and } \quad 
\sqrt{-1}\langle t, \dbar D'_{g}t \rangle_g= 
%\langle t, -\sqrt{-1}\dbar D'_{g}t \rangle_g=
%\langle t, -\sqrt{-1}\Theta_{g}(D) t \rangle_g=
-\sqrt{-1}\Theta_{g}(D)|t|_g^2. 
\end{align*}
These equalities lead to the desired equality. 

Now we compute the negativity of the curvature 
$\sqrt{-1}\Theta_{g_{\e}}(D)$. 
By the above equality, we have 
\begin{align*}
\sqrt{-1}\Theta_{g_{\e}}(D) \geq \e \frac{1} {|t|^2_{g}+\e}\sqrt{-1}\Theta_{g}(D).  
\end{align*}
On the other hand, 
there exists a positive constant $C$ such that $\sqrt{-1}\Theta_{g}(D) \geq -C \omega$ on $X$ 
since $X$ is compact and $g$ is smooth on $X$. 
For the positive $(1,1)$-form $A_{\e}$ defined by 
$$
A_{\e}:= \e \frac{C} {|t|^2_{g}+\e} \omega \geq 0, 
$$
we have  
$$
\sqrt{-1}\Theta_{g_{\e}}(D) + A_{\e} \geq 0. 
$$
Then we can see that  
\begin{align*}
&\lla \sqrt{-1}\Theta_{g_{\e}h_{F}h_{M}}(D\otimes F \otimes M) 
\Lambda su, su \rra_{g_{\e}} \\ \leq& 
\lla (\sqrt{-1}\Theta_{g_{\e}h_{F}h_{M}}(D\otimes F \otimes M) + A_{\e})
\Lambda su, su \rra_{g_{\e}} \\ \leq& 
\sup_{X}|s|^2_{h_{M}}
\lla (\sqrt{-1}\Theta_{g_{\e}h_{F}h_{M}}(D\otimes F \otimes M) + A_{\e})
\Lambda u, u \rra_{g_{\e}} \\ \leq& 
\sup_{X}|s|^2_{h_{M}}
\lla (\sqrt{-1}\Theta_{g_{\e}h_{F}h_{M}}(D\otimes F \otimes M) + A_{\e})
\Lambda u, u \rra_{g_{0}}. 
\end{align*}
The first inequality is obtained from $A_{\e} \geq 0$, 
the second inequality is obtained from $\sqrt{-1}\Theta_{g_{\e}}(D)+A_{\e} \geq 0$, 
and the third inequality is obtained from $g_{\e} \leq g_{0}$. 
Further, by the assumption 
$\sqrt{-1}\Theta_{h_{F}}(F) \geq t \sqrt{-1}\Theta_{h_{M}}(M)$, 
we can see that 
\begin{align*}
\sqrt{-1}\Theta_{g_{\e}h_{F}h_{M}}(D\otimes F \otimes M) + A_{\e} 
&\leq \sqrt{-1}\Theta_{g_{\e}}(D) + A_{\e}+
(1+\frac{1}{t})\sqrt{-1}\Theta_{h_{F}}(F)\\
&\leq (1+\frac{1}{t}) (\sqrt{-1}\Theta_{g_{\e}}(D) + A_{\e}+
\sqrt{-1}\Theta_{h_{F}}(F)). 
\end{align*}
Here we used $\sqrt{-1}\Theta_{g_{\e}}(D)+A_{\e} \geq 0$ to obtain the second inequality. 
In summary, we have 
\begin{align*}
&\lla \sqrt{-1}\Theta_{g_{\e}h_{F}h_{M}}(D\otimes F \otimes M) 
\Lambda su, su \rra_{g_{\e}} \\ \leq& 
\sup_{X}|s|^2_{h_{M}}\, (1+\frac{1}{t})\, 
\lla (\sqrt{-1}\Theta_{g_{\e}}(D) + A_{\e}+
\sqrt{-1}\Theta_{h_{F}}(F))
\Lambda u, u \rra_{g_{0}}. 
\end{align*}
For the proof of Theorem \ref{key}, 
it is sufficient to estimate the order of the right hand side. 

\begin{prop}[cf. {\cite[Proposition 3.8]{Tak97}}]\label{f-1}
Under the above situation, we have 
$$
\lla (\sqrt{-1}\Theta_{g_{\e}}(D) +\sqrt{-1}\Theta_{h_{F}}(F))
\Lambda u, u \rra_{g_{0}}=0.
$$ 
\end{prop}
\begin{proof}
For simplicity, 
we put 
$$
w:=\sqrt{-1}\Theta_{g_{\e}h_{F}}(D\otimes F)\Lambda u =
(\sqrt{-1}\Theta_{g_{\e}}(D) +\sqrt{-1}\Theta_{h_{F}}(F))
\Lambda u. 
$$
Then it follows that $w \in L^{n,q}_{(2)}(D\otimes F)_{g_{0}}$ 
since the metric $g_{\e}h_{F}$ is smooth on $X$ and $u \in L^{n,q}_{(2)}(D\otimes F)_{g_{0}}$.
Indeed, 
there is a positive constant $C$ such that 
$-C \omega \leq \sqrt{-1}\Theta_{g_{\e}h_{F}}(D\otimes F) \leq C \omega$. 
Then we have $|w|_{g_{0}} \leq C q |u|_{g_{0}}$, and thus 
we can see that $w \in L^{n,q}_{(2)}(D\otimes F)_{g_{0}}$
by $u \in L^{n,q}_{(2)}(D\otimes F)_{g_{0}}$. 
Further, by $u \in \mathcal{H}^{n,q}_{g_{0}}(D \otimes F)$ and (\ref{j}), 
we have $\dbar u=0$ and $D'^*_{g_{\e}}u=0$. 
Therefore we obtain  
$$
\dbar \, \dbar^{*}_{g_{\e}} u = 
\sqrt{-1}\Theta_{g_{\e}h_{F}}(D\otimes F)\Lambda u=w
$$
from Proposition \ref{BKN}. 
In particular, 
we can see that 
$w \in \Ker \dbar \subset L^{n,q}_{(2)}(D\otimes F)_{g_{0}}$.

By (\ref{g}), we have the orthogonal decomposition 
$$
\Ker \dbar = \Image \dbar \oplus \Ker \phi \oplus (\Ker \phi)^{\perp} 
 \text{ in }  L^{n,q}_{(2)}(D\otimes F)_{g_{0}}, 
$$
and thus $w$ can be decomposed as follows\,: 
$$
w=w_{1} + w_{2} + w_{3}\text{ for some } 
w_{1} \in \Image \dbar,  
w_{2} \in \Ker \phi 
\text{, and } w_{3} \in (\Ker \phi)^{\perp}. 
$$
Since we have $u \in (\Ker \phi)^{\perp}$ by the construction of $u$, 
we obtain $\lla w, u \rra_{g_{0}}=\lla w_{3}, u \rra_{g_{0}}$. 
It is sufficient for the proof to show that $w_{3}$ is zero. 
It follows that  $\dbar^{*}_{g_{\e}} u \in L^{n,q}_{(2)}(D\otimes F)_{g_{\e}}$
since $\dbar^{*}_{g_{\e}} u$ is smooth on $X$. 
(Note that we do not know whether 
$\dbar^{*}_{g_{\e}} u \in L^{n,q}_{(2)}(D\otimes F)_{g_{0}}$.) 
By combining with $\dbar \, \dbar^{*}_{g_{\e}} u =w$, 
we can conclude that 
$$
w_{2} + w_{3} = w - w_{1} \in \Image \dbar \subset 
L^{n,q}_{(2)}(D\otimes F)_{g_{\e}}, 
$$ 
and thus we obtain $w_{2} + w_{3} =w - w_{1} \in \Ker \phi$. 
In particular we can see $w_{3}=0$. 
This completes the proof. 
\end{proof}

Finally we prove the following proposition.

\begin{prop}\label{f-2}
Under the above situation, we have 
$$
\lla A_{\e}\Lambda u, u \rra_{g_{0}}=O(\e (- \log \e)). 
$$
\end{prop}
\begin{proof}
By Remark \ref{Lef-rem} (which is obtained from Theorem \ref{Lef}), 
we see that $|u|_{g_{0}}$ is a bounded function on $X$. 
By the definition of $A_{\e}$, 
we can easily see that 
\begin{align*}
\lla A_{\e}\Lambda u, u \rra_{g_{0}} &= 
\e \int_{Y} \frac{Cq} {|t|^2_{g}+\e} |u|^2_{g_{0}}\, dV_{\omega} \leq 
\e \sup_{X}|u|^2_{g_{0}} \int_{Y} \frac{Cq} {|t|^2_{g}+\e} \, dV_{\omega}. 
\end{align*}
By Lemma \ref{comp}, we obtain the conclusion. 
\end{proof}

\begin{rem}\label{important}
The integral in Lemma \ref{comp} naturally appears 
when we prove Proposition \ref{bdd} and Proposition \ref{f-2}, 
but the reasons why the integral appears are different. 
The integral in Proposition \ref{bdd} comes from the definition of $g_{\e}$.  
On the other hand, the same integral comes 
from the curvature of $g_{\e}$ when we prove Proposition \ref{f-2}. 
\end{rem}

By Proposition \ref{bdd}, Proposition \ref{f-2}, and inequality (\ref{i}), 
we complete the proof of Theorem \ref{key}.  
\end{proof}

\subsection{Proof of Theorem \ref{main}}\label{Sec3-3}

In this subsection, 
we prove that Theorem \ref{key} leads to Theorem \ref{main}. 
In particular, 
Conjecture \ref{F-conj} is affirmatively solved for plt pairs 
(see Corollary \ref{main-co}).  

\begin{theo}[=Theorem \ref{main}]\label{mainn}
Let $D$ be a simple normal crossing divisor on a compact K\"ahler manifold $X$. 
Let $F$ $($resp. $M$$)$ be a $($holomorphic$)$ line bundle on $X$ 
with a smooth hermitian metric $h_{F}$  $($resp. $h_{M}$$)$ 
such that 
\begin{equation*}
\sqrt{-1}\Theta_{h_M}(M)\geq 0 \  \text{ and }\ 
\sqrt{-1}(\Theta_{h_F}(F)-t \Theta 
_{h_M}(M))\geq 0 
\ \text{ for some }t>0. 
\end{equation*}
We assume that the pair $(X,D)$ is a plt pair. 
Let $s$ be a $($holomorphic$)$ section of $M$  
such that the zero locus $s^{-1}(0)$ 
contains no lc centers of the lc pair $(X,D)$.
Then, the multiplication map induced by the tensor product with $s$
\begin{equation*}
H^q(X, K_X\otimes D \otimes F)
\xrightarrow{\otimes s } 
H^q(X, K_X\otimes D \otimes F\otimes M)
\end{equation*} 
is injective for every $q$. 
\end{theo}

\begin{proof}
Let $D=\sum_{i \in I}D_{i}$ be the irreducible decomposition of $D$. 
We remark that $D_{i} \cap D_{j} = \emptyset$ for $i \not = j$ 
since $(X,D)$ is a plt pair. 
For every $i \in I$, 
we consider the long exact sequence 
induced by the standard short exact sequence\,: 
\begin{align}\label{h}
\vcenter{ \xymatrix{
&\ar[d] & \ar[d]\\
&H^q(X,\mathcal{O}_{X}(K_{X}\otimes F \otimes \hat{D_{i}}))
\ar[d]^-{\Phi_{D_{i}}} \ar[r]^-{\otimes s} 
&H^q(X, \mathcal{O}_{X}(K_{X}  \otimes F \otimes \hat{D_{i}}\otimes M))\ar[d]\\ 
&H^q(X, \mathcal{O}_{X}(K_{X}\otimes D \otimes F))
\ar[d]^-{r_i}\ar[r] ^-{\otimes s}
&H^q(X, \mathcal{O}_{X}(K_{X}\otimes D \otimes F\otimes M)) \ar[d]\\ 
&H^q(D_{i}, \mathcal{O}_{D_{i}}(K_{D_{i}}\otimes F \otimes \hat{D_{i}}))
\ar[d]\ar[r]^-{f_{i}:=\otimes s|_{D_{i}} } 
& H^q(D_{i}, \mathcal{O}_{D_{i}}(K_{D_{i}}\otimes F \otimes \hat{D_{i}}\otimes M))\ar[d]\\ 
& & 
}}
\end{align}
Here $\hat{D_{i}}$ is the divisor defined by 
$\hat{D_{i}}:= \sum_{k\in I, k \not = i} D_{k}$ and 
$f_{i}$ is the multiplication map induced by the tensor product 
with the restriction $s|_{D_{i}}$ of $s$ to $D_{i}$. 
Further $\Phi_{D_{i}}$ is the map induced by the natural inclusion 
$\OX_{X} \hookrightarrow \OX_{X}(D_{i})$ and 
$r_{i}$ is the map induced by the restriction map 
$\OX_{X} \to \OX_{D_{i}}$. 
Note that we used the adjunction formula 
$\mathcal{O}_{D_{i}}(K_{X}\otimes D_{i})=\mathcal{O}_{D_{i}}(K_{D_{i}})$. 

\begin{rem}\label{after}
By the assumption $D_{i} \cap D_{j} = \emptyset$, 
we actually have 
$\mathcal{O}_{D_{i}}(K_{D_{i}}\otimes F \otimes \hat{D_{i}})=
\mathcal{O}_{D_{i}}(K_{D_{i}}\otimes F )$, 
but we used the notation 
$\mathcal{O}_{D_{i}}(K_{D_{i}} \otimes F \otimes \hat{D_{i}})$ 
for Observation \ref{toward-conj}. 
\end{rem}

Let $\alpha$ be a cohomology class in 
$H^q(X, \mathcal{O}_{X}(K_{X}\otimes D \otimes F))$ 
such that $s \alpha =0 \in 
H^q(X, \mathcal{O}_{X}(K_{X}\otimes D \otimes F\otimes M))$. 
The above commutative diagram implies that $f_{i}( r_{i}(\alpha))=0$. 
Note that we have 
$\mathcal{O}_{D_{i}}(K_{D_{i}}\otimes F \otimes \hat{D_{i}})=
\mathcal{O}_{D_{i}}(K_{D_{i}}\otimes F )$
by the assumption $D_{i} \cap D_{j} = \emptyset$. 
The restriction $\mathcal{O}_{D_{i}}(F)$ is a semi-positive line bundle on $D_{i}$ 
since $F$ is semi-positive, 
and further the restriction $s|_{D_{i}}$ is non-zero since 
the zero locus $s^{-1}(0)$ does not contain $D_{i}$ 
by the assumption.  
In particular $\mathcal{O}_{D_{i}}(F)$ and $s|_{D_{i}}$ satisfy 
the assumptions of Enoki's injectivity theorem, 
and thus the multiplication map $f_{i}$ is injective. 
Therefore we obtain $r_{i}(\alpha)=0$ for every $i \in I$. 

We have the following exact sequence\,:
\begin{equation*}
\cdots \to 
H^{q}(X, K_{X} \otimes F) 
\xrightarrow{\Phi_{D} } 
H^{q}(X, K_{X} \otimes D \otimes F )
\xrightarrow{r_{D}} 
H^{q}(X, \mathcal{O}_{D}(K_{X} \otimes D \otimes F) )
\to \cdots, 
\end{equation*}
where $r_{D}$ is the map induced 
by the restriction map $\OX_{X} \to \OX_{D}$. 
On the other hand, we have 
$$
H^{q}(X, \mathcal{O}_{D}(K_{X} \otimes D \otimes F) )=
\bigoplus_{i \in I}H^{q}(X, \mathcal{O}_{D_{i}}(K_{D_{i}} \otimes F) )
$$
by the assumption $D_{i} \cap D_{j} = \emptyset$. 
Then we can easily check $r_{D}(\alpha)=0$ by the above exact sequence 
since we have $r_{i}(\alpha)=0$ for every $i \in I$. 
Therefore Theorem \ref{key} leads to the desire conclusion $\alpha=0$ 
of Theorem \ref{main}. 
\end{proof}

Corollary \ref{main-co} is easily obtained from Theorem \ref{main}. 
Indeed, the hermitian line bundle $(M, h_{M}):=(F^{m}, h_{F}^{m})$ 
satisfies the assumption 
$\sqrt{-1}\Theta_{h_{F}}(F) \geq (1/m)\sqrt{-1}\Theta_{h_{M}}(M)$ in Theorem \ref{main}.

\subsection{Open problems related to Conjecture \ref{F-conj}}\label{Sec3-4}

In this subsection, 
we give several open problems related to Conjecture \ref{F-conj}. 

We first consider a generalization of Theorem \ref{Lef}. 
For Conjecture \ref{F-conj}, 
our formulation of Theorem \ref{Lef} seems to be enough, 
but it is an interesting problem 
to remove the technical assumption in Theorem \ref{Lef}. 

\begin{prob}\label{P-Lef}
Consider the same situation as in Theorem \ref{Lef}. 
Then can we remove the assumption that $h$ is smooth on 
a non-empty Zariski open set? 
\end{prob}

In \cite{DPS01}, 
it has been shown that the map defined by 
the wedge product $\omega^{q}\wedge \bullet$
$$
H^{0}(X, \Omega_{X}^{n-q}\otimes G \otimes \I{h}) 
\xrightarrow{\ \omega^{q}\wedge \bullet \ }
H^{q}(X, K_{X}\otimes G \otimes \I{h})
$$
is surjective without the assumption that $h$ is smooth on 
a non-empty Zariski open set. 
We remark that Problem \ref{P-Lef} leads to the above result. 
Problem \ref{P-Lef} can be seen as a refinement of \cite[Theorem 0.1]{DPS01}.

The following problem may give a strategy to solve Conjecture \ref{F-conj}. 

\begin{prob}\label{to-conj}
Let $D$ be a simple normal crossing divisor on a compact K\"ahler manifold $X$ 
and $F$ be a semi-positive line bundle on $X$. 
Let $s$ be a $($holomorphic$)$ section of  
$\mathcal{O}_{D}(F^m)$ restricted to the $($possibly non-irreducible$)$ variety $D$. 
Then, is the following multiplication map injective?
\begin{equation*}
H^{q}(D, \mathcal{O}_{D}(K_{X} \otimes D \otimes F)) 
\xrightarrow{\otimes s} 
H^{q}(D, \mathcal{O}_{D}({K_{X} \otimes D \otimes F^{m+1}})). 
\end{equation*}
\end{prob}

\begin{rem}\label{rem-to-conj}
When $X$ is a smooth projective variety and $F$ is a semi-ample line bundle on $X$, 
Problem \ref{to-conj} has already proved. 
By Theorem \ref{key} and the proof of Theorem \ref{main}, 
we can see that if Problem \ref{to-conj} is affirmatively solved, 
we can prove Conjecture \ref{to-conj}. 
\end{rem}

Finally, in order to clarify what is needed for Conjecture \ref{F-conj}, 
we attempt to prove Conjecture \ref{F-conj} 
by the induction on $n=\dim X$. 
\begin{step}[Observation for Conjecture \ref{F-conj}]\label{toward-conj}
In the case $D=0$, Conjecture \ref{F-conj} is the same as Enoki's injectivity theorem, 
and thus we may assume that $D \not= 0$. 
When $n$ is one, the conclusion of Conjecture \ref{F-conj} is obvious 
since $D\otimes F$ is ample. 
Hence we may assume that Conjecture \ref{F-conj} holds 
for compact K\"ahler manifolds of dimension $(n-1)$.

We consider the commutative diagram (\ref{h}) in the proof of Theorem \ref{main}. 
We remark that the pair $(D_{i}, \hat{D_{i}})$ is an lc pair. 
Since the zero locus $s^{-1}(0)$ contains no lc centers of $(X,D)$, 
we can show that the restriction $s|_{D_{i}}$ contains no lc centers of $(D_{i}, \hat{D_{i}})$. 
Further the restriction $\mathcal{O}_{D_{i}}(F)$ is a semi-positive line bundle on $D_{i}$.  
Therefore the multiplication map $f_{i}$ in (\ref{h}) is injective 
by the induction hypothesis.

For a cohomology class $\alpha$ in 
$H^q(X, \mathcal{O}_{X}(K_{X}\otimes D \otimes F))$ 
such that $s \alpha =0 \in 
H^q(X, \mathcal{O}_{X}(K_{X}\otimes D \otimes F\otimes M))$, 
we have $f_{i}(r_{i}(\alpha))=0$. 
Then it follows that  $r_{i}(\alpha) =0$ for every $i \in I$ 
since $f_{i}$ is injective. 
In the case of plt pairs, 
we have $D_{i} \cap D_{j} = \emptyset$ for $i \not = j$. 
Then we can obtain $r_{D}(\alpha)=0$ from $r_{i}(\alpha) =0$  
(see the proof of Theorem \ref{main}). 
If we can show that $r_{D}(\alpha)=0$ in the case of lc pairs, 
Conjecture \ref{F-conj} is affirmatively solved by Theorem \ref{key}. 
However we do not know whether we can conclude $r_{D}(\alpha)=0$ from $r_{i}(\alpha) =0$ 
in this case. 
\end{step}

\newpage
%%%%%%%%%%%%%%%%%%%%%%


\begin{thebibliography}{n}

\bibitem[Amb03]{Amb03}
F. Ambro,   
\textit{Quasi-log varieties}, 
Tr. Mat. Inst. Steklova {\bf{240}} (2003), 
Biratsion. Geom. Linein. Sist. Konechno Porozhdennye Algebry, 220--239. 


\bibitem[Amb14]{Amb14}
F. Ambro,   
\textit{An injectivity theorem}, 
Compos. Math. {\bf{150}} (2014), no. 6, 999--1023. 


\bibitem[Dem-a]{Dem-book}
J.-P. Demailly,   
\textit{Complex analytic and differential geometry}, 
Lecture Note on the web page of the author.

\bibitem[Dem-b]{Dem-L2}
J.-P. Demailly,  
\textit{$L^2$ estimates for the $\dbar$-operator on complex manifolds}, 
Lecture Note on the web page of the author. 

%\bibitem[Dem82]{Dem82}
%J.-P. Demailly, 
%\textit{Estimations $L^{2}$ pour l'op\'erateur $\overline{\partial}$ d'un 
%fibr\'e vectoriel holomorphe semi-positif au-dessus d'une vari\'et\'e k\"ahl\'erienne compl\`ete}, 
%Ann. Sci. \'Ecole Norm. Sup (4). {\bf{15}} (1982), no. 3, 457--511. 

\bibitem[DPS01]{DPS01}
J.-P. Demailly, T. Peternell, and M. Schneider,   
\textit{Pseudo-effective line bundles on compact K\"ahler manifolds}, 
Internat. J. Math. {\bf{6}} (2001), no. 6, 689--741. 

\bibitem[Eno90]{Eno90}
I. Enoki, 
\textit{Kawamata-Viehweg vanishing theorem for compact K\"ahler manifolds}, 
Einstein metrics and Yang-Mills connections (Sanda, 1990), 59--68. 

\bibitem[EV]{EV92}
H. Esnault, and E. Viehweg, 
\textit{Lectures on vanishing theorems},  
DMV Seminar, {\bf{20}}. Birkh\"auser Verlag, Basel, (1992). 


\bibitem[FST11]{FST11}
O. Fujino, K. Schwede, and S. Takagi, 
\textit{Supplements to non-lc ideal sheaves},  
Higher dimensional algebraic geometry, 
RIMS K$\rm{\hat{o}}$ky$\rm{\hat{u}}$roku Bessatsu, 
B{\bf{24}} (2011), 1--46, 



%\bibitem[Fuj09]{Fuj09}
%O. Fujino, 
%\textit{On injectivity, vanishing and torsion-free theorems for algebraic varieties}, 
%Proc. Japan Acad. Ser. A Math. Sci. {\bf{85}} (2009), no. 8, 95--100.

\bibitem[Fuj10]{Fuj10}
O. Fujino, 
\textit{Theory of non-lc ideal sheaves:\,basic properties}, 
Kyoto J. Math. {\bf{50}} (2010), no. 2, 225--245. 


\bibitem[Fuj11]{Fuj11} 
O. Fujino, 
\textit{Fundamental theorems for the log minimal model program}, 
Publ. Res. Inst. Math. Sci. {\textbf{47}} (2011), no. 3, 727--789. 

\bibitem[Fuj12a]{Fuj12a}
O. Fujino, 
\textit{A transcendental approach to Koll\'ar's injectivity theorem}, 
Osaka J. Math. {\bf{49}} (2012), no. 3, 833--852.

\bibitem[Fuj12b]{Fuj12b} 
O. Fujino, 
\textit{Vanishing theorems}, 
to appear in Adv. Stud. Pure Math,  
arXiv:1202.4200v2. 

\bibitem[Fuj13a]{Fuj13a}
O. Fujino, 
\textit{A transcendental approach to Koll\'ar's injectivity theorem} I\hspace{-.1em}I, 
J. Reine Angew. Math. {\bf{681}} (2013), 149--174. 

\bibitem[Fuj13b]{Fuj13b}
O. Fujino, 
\textit{Injectivity theorems}, 
to appear in Adv. Stud. Pure Math, arXiv:1303.2404v3. 

%\bibitem[Fuj14a]{Fuj14a} 
%O. Fujino, 
%\textit{Fundamental theorems for semi log canonical pairs}, 
%Algebr. Geom. {\textbf{1}} (2014), no. 2, 194--228.

%\bibitem[Fuj14b]{Fuj14b}
%O. Fujino, 
%\textit{Foundation of the minimal model program}, 
%preprint (2014). 

\bibitem[Fuj15a]{Fuj15} 
O. Fujino, 
\textit{Kodaira vanishing theorem for log-canonical and semi-log-canonical pairs}, 
Proc. Japan Acad. Ser. A Math. Sci. {\textbf{91}} (2015), no. 8, 
112--117.

\bibitem[Fuj15b]{Fuj15b} 
O. Fujino, 
\textit{On semipositivity, injectivity, and vanishing theorems}, 
Preprint, arXiv:1503.06503v3. 

\bibitem[FM16]{FM16}
O.  Fujino, and S. Matsumura, 
\textit{Injectivity theorem for pseudo-effective line bundles and its applications}, 
Preprint, arXiv:1605.02284v1.  

\bibitem[GM14]{GM13}
Y. Gongyo, and S. Matsumura, 
\textit{Versions of injectivity and extension theorems}, 
to appear in Ann. Sci. Ecole Norm. Sup, arXiv:1406.6132v2. 

\bibitem[Kol86a]{Kol86a}
J. Koll\'ar, 
\textit{Higher direct images of dualizing sheaves} I, 
Ann. of Math. (2) {\bf{123}} (1986), no. 1, 11--42.  

%\bibitem[Kol86b]{Kol86b}
%J. Koll\'ar, 
%\textit{Higher direct images of dualizing sheaves} I\hspace{-.1em}I, 
%Ann. of Math. (2) {\bf{124}} (1986), no. 1, 171--202.   

\bibitem[Laz]{Laz}
R. Lazarsfeld,  
\textit{Positivity in Algebraic Geometry} I-I\hspace{-.1em}I, 
A Series of Modern Surveys
in Mathematics, {\bf{48, 49}}
Springer Verlag, Berlin, (2004).

%\bibitem[Mat14]{Mat14}
%S. Matsumura, 
%\textit{A Nadel vanishing theorem via injectivity theorems}, 
%Math. Ann. {{\bf 359}} (2014), no. 3--4, 785--802. 

\bibitem[Mat13]{Mat16}
S. Matsumura, 
\textit{An injectivity theorem with multiplier ideal sheaves 
of singular metrics with transcendental singularities}, 
to appear in J. Algebraic Geom, arXiv:1308.2033v4.


\bibitem[Mat15]{Mat15}
S. Matsumura, 
\textit{Injectivity theorems with multiplier ideal sheaves and their applications}, 
Complex analysis and geometry, 241--255, 
Springer Proc. Math. Stat., {\textbf{144}}, Springer, Tokyo, (2015).



\bibitem[Mat16]{Mat16b}
S. Matsumura, 
\textit{Injectivity theorems with multiplier ideal sheaves 
for higher direct images under K\"ahler morphisms}, 
Preprint, arXiv:1607.05554v1. 

\bibitem[Ohs04]{Ohs04}
T. Ohsawa, 
\textit{On a curvature condition that implies a cohomology injectivity theorem of Koll\'ar-Skoda type},  
Publ. Res. Inst. Math. Sci. {\bf{41}} (2005), 
no. 3, 565--577.

\bibitem[Take95]{Tak95}
K. Takegoshi,  
\textit{Higher direct images of canonical sheaves tensorized with semi-positive vector bundles by proper K\"ahler morphisms},  
Math. Ann. {\bf{303}} (1995), no. 3, 389--416.

\bibitem[Tak97]{Tak97}
K. Takegoshi, 
\textit{On cohomology groups of nef line bundles tensorized with multiplier ideal sheaves 
on compact K\"ahler manifolds},  
Osaka J. Math. {\bf 34} (1997), no. 4, 783--802.


\bibitem[Tan71]{Tan71} 
S. G. Tankeev, 
\textit{On $n$-dimensional canonically polarized varieties 
and varieties of fundamental type}, 
Math. USSR-Izv. {\bf{5}} (1971), no. 1, 29--43. 

\end{thebibliography}
\end{document}